\documentclass[a4paper,12pt]{article}
\usepackage{amsmath,amssymb,graphicx,subfigure}
\usepackage{algorithm}
\usepackage{algorithmic}
\usepackage{caption}
\usepackage{amssymb}
\usepackage{latexsym}
\usepackage{amsmath}
\usepackage{indentfirst}
\usepackage{graphicx}
\usepackage{subfigure}
\usepackage{float}
\usepackage{amsthm}

\title{Monochromatic cycles in $2$-edge-colored bipartite graphs with large minimum degree\thanks{This work was supported by NSFC (Grant Nos. 11931002 and 12371327).}}
\author{\small{Yiran Zhang\thanks{Email address: zhangyiran1125@163.com}~,
Yuejian Peng\thanks{Corresponding author. Email address: ypeng1@hnu.edu.cn}~}\\
{\small School of Mathematics, Hunan University,}\\
{\small Changsha, Hunan 410082, People's Republic of China}\\\makeatletter}
\newcommand\tabcaption{\def\@captype{table}\caption}
\date{}\makeatother
\newtheorem{theorem}{Theorem}[section]
\newtheorem{defi}{Definition}[section]
\newtheorem{lemma}[theorem]{Lemma}

\newtheorem{conjecture}{Conjecture}[section]
\newtheorem{claim}{Claim}
\newtheorem{fact}{Fact}[section]
\newtheorem{const}{Construction}[section]
\begin{document}
\maketitle
\begin{abstract}
For graphs $G$, $G_1$ and $G_2$, we write $G\longmapsto(G_1, G_2)$ if each red-blue-edge-coloring of $G$ yields a red $G_1$ or a blue $G_2$. The {\em Ramsey number} $r(G_1, G_2)$ is the minimum number $n$ such that the complete graph $K_n\longmapsto(G_1, G_2)$. There is an interesting phenomenon that  for some  graphs $G_1$ and $G_2$ there is a number $0<c<1$ such that for any graph $G$ of order  $r(G_1, G_2)$ with minimum degree $\delta(G)>c|V(G)|$, $G\longmapsto(G_1, G_2)$. When we focus on bipartite graphs,  the  {\em  bipartite Ramsey number} $br(G_1, G_2)$ is the minimum number $n$ such that the complete bipartite graph $K_{n, n}\longmapsto(G_1, G_2)$. Previous known related results on  cycles are  on the diagonal case ($G_1=G_2=C_{2n}$). In this paper, we obtain an asymptotically tight bound for all off-diagonal cases, namely, we determine an asymptotically tight bound on the minimum degree of a balanced bipartite graph $G$ with order $br(C_{2m}, C_{2n})$ in each part such that $G\longmapsto(C_{2m}, C_{2n})$. We show that for every $\eta>0$, there is an integer $N_0>0$ such that for any $N>N_0$ the following holds: Let $\alpha_1>\alpha_2>0$ such that $\alpha_1+\alpha_2=1$. Let $G[X, Y]$ be a balanced bipartite graph on $2(N-1)$ vertices with minimum degree  $\delta(G)\geq(\frac{3}{4}+3\eta)(N-1)$. Then for each red-blue-edge-coloring of $G$, either there exist red even cycles of each length in $\{4, 6, 8, \ldots, (2-3\eta^2)\alpha_1N\}$, or there exist blue even cycles of each length in $\{4, 6, 8, \ldots, (2-3\eta^2)\alpha_2N\}$. A construction is given to show the bound $\delta(G)\geq(\frac{3}{4}+3\eta)(N-1)$ is asymptotically tight. Furthermore, a stability result is given.

%\noindent{\bf AMS subject classification 2010:} 05C15, 05C38, 05C40.
\end{abstract}

\section{Introduction}

For $r\geq2$ and graphs $G$, $G_1$, $G_2$,  we write $G\longmapsto(G_1, G_2)$ if each 2-edge-coloring of $G$ yields a monochromatic $G_i$ for some $i\in[2]$. The {\em Ramsey number} $r(G_1, G_2)$ is the minimum number $n$ such that the complete graph $K_n\longmapsto(G_1, G_2)$. If $G_1=G_2$, we say that $G$ arrows $G_1$, which we call {\em diagonal case}. The Ramsey numbers of cycles was determined independently by Bondy and Erd\H{o}s \cite{Bond1973}, Faudree and Schelp \cite{FaSch1974}, and Rosta\cite{Rosta1, Rosta2}. These results showed that for $m\geq n\geq3$,
\[r(C_m, C_n)=\begin{cases}
2m-1, & if~n~is~odd~and~(m,~n)\neq(3,3), \\
m+\frac{n}{2}-1, & if~m~and~n~are~even~and~(m,~n)\neq(4,4), \\
\max\{m+\frac{n}{2}-1,~2n-1\}, & if~m~is~odd~and~n~is~even.
\end{cases}\]

A {\em connected $k$-matching} in a graph $G$, denoted by $CM_k$, is a matching with $k$ edges lying in a component of $G$, where a component of $G$ is a maximal connected subgraph of $G$. In \cite{Luc1999}, {\L}uczak firstly employed the following approach to show that $r(C_n, C_n, C_n)\leq(4+o(1))n$ for large $n$. First show the existence of a large monochromatic connected matching in the reduced graph  obtained by applying Szem\'{e}redi's Regularity Lemma, then guaranteed by the Regularity Lemma, this monochromatic connected matching in the reduced graph can be extended to a long monochromatic cycle in the original graph. Letzter \cite{letz} showed further that obtaining asymptotic Ramsey numbers of cycles can be
 reduced to determining Ramsey numbers of  monochromatic connected matchings.

Schelp \cite{Sch2012} observed that for some sparse graphs $G$, such as cycles and paths, a  graph $H$ of order $r(G, G)$ with large minimum degree also arrows $G$.
In 2007, Nikiforov and Schelp showed the following  result for cycles.

For sets $X$ and $Y$, let $X\uplus Y$ denote the disjoint union of $X$ and $Y$.

\begin{theorem}[Nikiforov and Schelp \cite{NS2008}]\label{NS}
If n is sufficiently large and G is a graph of order $2n-1$ with minimum degree $\delta(G)\geq(2-10^{-6})n$,
then for each {\rm2}-edge-coloring $E(G)=E(R)\uplus E(B)$, either $C_t\subset E(R)$ for all $t\in[3, n]$ or $C_t\subset E(B)$ for all $t\in[3, n]$.
\end{theorem}

Note that $r(C_n, C_n)=2n-1$ if $n\geq5$ is odd. And results  as in Theorem \ref{NS} are more interesting and challenging when we require the order of $G$ starting from $r(C_n, C_n)$.

In 2010, Li, Nikiforov and Schelp proposed the following conjecture.
\begin{conjecture}[Li, Nikiforov and Schelp \cite{LNS2010}]\label{Nik}
Let $n\geq4$ and let $G$ be an $n$-vertex graph with minimum degree $\delta(G)>\frac{3}{4}n$. If $E(G)=E(B)\uplus E(R)$ is a {\rm2}-edge-coloring of G, then either $C_k\subseteq E(B)$ or $C_k\subseteq E(R)$ for all $k\in[4,\lceil\frac{n}{2}\rceil]$.
\end{conjecture}

Li, Nikiforov and Schelp \cite{LNS2010} also showed that if $n$ is large enough and $k\in[4,(\frac{1}{8}-o(1))n]$, $G$ in Conjecture \ref{Nik} arrows $C_k$. Benevides, {\L}uczak, Scott, Skokan and White \cite{BLS2012} proved that for large n, Conjecture \ref{Nik} is correct except one special 2-edge-coloring of $G$, and they proposed the following conjecture.

\begin{conjecture}[Benevides, ${\L}$uczak, Scott, Skokan and White \cite{BLS2012}]\label{Ben}
Let $G$ be an $n$-vertex graph with minimum degree  $\delta(G)\geq\frac{3}{4}n$, where $n=3t+r$, $r\in\{0, 1, 2\}$. Then each $2$-edge-coloring of $G$ yields a monochromatic cycle of length at least $2t+r$.
\end{conjecture}

In \cite{Sch2012}, Schelp formulated  the following question: for which graphs $H$ there exists a constant $c\in(0,1)$ such that for any graph $G$ of order  $r(H, H)$ with $\delta(G)>c|V(G)|$, $G\longmapsto(H, H)$. Meanwhile, Schelp posed the following conjecture.

\begin{conjecture}[Schelp \cite{Sch2012}]\label{Schelp}
Let $t=r(P_n, P_n)$ with n large. If $G$ is a graph of order $t$ with minimum degree  $\delta(G)>\frac{3}{4}t$, then $G\longmapsto (P_n, P_n)$.
\end{conjecture}

Gy\'{a}rf\'{a}s and S\'{a}rk\"{o}zy \cite{SMM2012} determined the Ramsey number $r(S_t, n_1K_2, n_2K_2)$, combining with Szem\'{e}redi's Regularity Lemma, they obtained an asymptotic form of Conjecture \ref{Schelp}. Balogh, Kostochka, Lavrov, and Liu \cite{BKLL2022} confirmed Conjecture \ref{Ben}  for large $n$, and Conjecture \ref{Schelp}  for all even paths. In fact, they proved the following stronger result.

\begin{theorem}[Balogh, Kostochka, Lavrov, and Liu \cite{BKLL2022}]\label{BK}
There exists an integer $n_0$ with the following property. Let $n=3t+r>n_0$, where $r\in\{0,1,2\}$.
Let $G$ be an $n$-vertex graph with minimum degree $\delta(G)\geq\frac{3n-1}{4}$. Then for any $2$-edge-coloring of $G$, either there are cycles of every length in $\{3, 4, \ldots, 2t+r\}$ of the same color, or there are cycles of every even length in $\{4, 6, \ldots, 2t+2\}$ of the same color.
\end{theorem}

We may change the host graph from a complete graph to a complete bipartite graph. The {\em bipartite Ramsey number} $br(G_1, G_2, \ldots, G_r)$ is the minimum number $N$ such that the complete bipartite graph $K_{N, N}\longmapsto(G_1, G_2, \ldots, G_r)$. If $G_1=G_2=\dots=G_k=G$, then simplify it as $br^k(G)$. Let $G$ be a bipartite graph with partition $V_1\uplus V_2$. For $X\subseteq V_1$, $Y\subseteq V_2$, let $G[X,Y]$ denote the bipartite subgraph of $G$ induced by $X\uplus Y$. We call
$G[V_1, V_2]$ a {\em balanced bipartite graph} if $|V_1|=|V_2|$. The study of bipartite Ramsey number was initiated in the early 1970s by Faudree and Schelp \cite{FS}, and Gy\'arf\'as and Lehel \cite{GL}. They determined the bipartite ramsey numbers of paths. Applying {\L}uczak's method, i.e., combining their result on paths and Szem\'{e}redi's  regularity lemma, one can obtain the asymptotic values of bipartite Ramsey numbers of cycles, one can also see the paper of Shen, Lin and Liu \cite{SLL} for a more general result including cycles.
Buci\'c, Letzter and Sudakov showed that $br^3(C_{2n})=(3+o(1))n$ in \cite{BLS}, and $br^k(C_{2n})\le (2k-3+o(1))n$ for $k\ge 5$ and $br^4(C_{2n})=(5+o(1))n$ in \cite{BLS1}. Liu and Peng \cite{LiuP} gave an asymptotic value of  $br(C_{2\lfloor\alpha_1n\rfloor}, \dots, C_{2\lfloor\alpha_rn\rfloor})$ when $r\ge 3$, $\alpha_1, \alpha_2>0$ and $\alpha_{j+2}\ge [(j+2)!-1]\sum_{i=1}^{j+1}\alpha_i$ for $1\le j\le r-2$. Luo and Peng \cite{LP} gave an asymptotic value of $br(C_{2\lfloor\alpha_1n\rfloor}, C_{2\lfloor\alpha_2n\rfloor}, C_{2\lfloor\alpha_3n\rfloor})$ for any $\alpha_1, \alpha_2, \alpha_3>0$.
Recently, DeBiasio and Krueger \cite{DBK2020} studied a bipartite version of Schelp's question and they obtained the following.
\begin{theorem}[DeBiasio and Krueger \cite{DBK2020}]\label{DK}
Let $G$ be a balanced bipartite graph of order $2n$. If $\delta(G)\geq\frac{3}{4}n$, then $G\longmapsto (CM_{n \over 2}, CM_{n \over 2})$.
\end{theorem}

Combining Theorem \ref{DK} with {\L}uczak's method, they obtained the following result.

\begin{theorem}[DeBiasio and Krueger \cite{DBK2020}]\label{DKK}
For all real numbers $\gamma$, $\eta$ with $0\leq 32\sqrt[4]{\eta} < \gamma \leq \frac{1}{4}$, there exists $n_0$ such that if $G$ is a balanced bipartite graph on $2n\geq2n_0$ vertices with minimum degree  $\delta(G)\geq(\frac{3}{4}+\gamma)n$, then in every $2$-edge-coloring of G, either there exists a monochromatic cycle on at least $(1+\eta)n$ vertices, or there exist a monochromatic path on at least $2\lceil\frac{n}{2}\rceil$ vertices and a monochromatic cycle on at least $2\lfloor\frac{n}{2}\rfloor$ vertices.
\end{theorem}

Note that  previous studies on Schelp's question on cycles  (as Theorem \ref{NS}, Theorem \ref{BK}, Theorem \ref{DKK}) are basically on the diagonal case. The conclusions in Theorem \ref{NS} and Theorem \ref{BK} are pancyclic, but if we look at the longest cycle in Theorem \ref{BK} (for example), it basically says that $G\longrightarrow(C_{2t+2}, C_{2t+2})$ and the order of $G$ is $3t+r$ different from $r(C_{2t+2}, C_{2t+2})$ by at most 2.

For  bipartite Ramsey numbers of cycles, Yan and Peng \cite{Yan2021} recently showed that for $m,~n\geq5$,
\begin{equation}\label{brc}
br(C_{2m}, C_{2n})=\begin{cases}
m+n-1, & m\neq n, \\
m+n, & m=n.
\end{cases}
\end{equation}
If $\min\{m,~n\}\leq4$, (\ref{brc}) also holds by the results of Beineke and Schwenk \cite{BS1976}, Zhang and Sun \cite{ZS2011}, Zhang, Sun and Wu \cite{ZS2013}, and Gholami and Rowshan \cite {Gho2021}.

In this paper, we study the minimum degree version for off-diagonal case of cycles. For $m\neq n$, what is the tight bound on $\delta(G)$ such that $G\longmapsto(C_{2m}, C_{2n})$ for any balanced bipartite graph $G$ of order $2br(C_{2m}, C_{2n})=2(m+n-1)$? By the method introduced by {\L}uczak \cite{Luc1999} (see further development by Letzter\cite{letz}), if we can obtain bipartite Ramsey numbers for connected matchings, then  we can obtain   asymptotic bipartite Ramsey numbers for cycles.  In this paper, we prove that for $m\neq n$, if $G$ is a balanced bipartite graph of order $2br(CM_m, CM_n)=2(m+n-1)$ with $\delta(G)>\frac{3}{4}(m+n-1)$, then $G\longmapsto(CM_m, CM_n)$. Applying {\L}uczak's method, we obtain an asymptotic result of Schelp's question for  cycles and we state our main results below.
We first show the following theorem.

\begin{theorem}\label{1} Let $G[V_1, V_2]$ be a balanced bipartite graph on $2(m+n-1)$ vertices with minimum degree  $\delta(G)>\frac{3}{4}(m+n-1)$, where $m>n$. Then $G\longrightarrow(CM_m, CM_n)$.
\end{theorem}

Note that $br(CM_m, CM_n)=m+n-1$ for $m\neq n$. Let $K_{N, N}$ have bipartition $X\uplus Y$, where $N=m+n-2$. Let $X_1\uplus X_2$ be a partition of $X$ with $|X_1|=m-1$ and $|X_2|=n-1$. Color all edges between $X_1$ and $Y$ in red, and all edges between $X_2$ and $Y$ in blue, then there is neither a red $CM_m$ nor a blue $CM_n$. Thus $br(CM_m, CM_n)\geq m+n-1$ for $m\neq n$. By equation (\ref{brc}), $br(CM_m, CM_n)\leq m+n-1$ for $m\neq n$.

The following construction shows that the minimum degree condition in Theorem \ref{1} is tight.

\begin{const}\label{a}
Let $X$ and $Y$ be disjoint sets with $m+n-1$ vertices, where $n<m<3n$. Partition $X$ into $\{X_i: i\in[4]\}$ and partition Y into $\{Y_i: i\in[4]\}$, such that $|X_i|=|Y_i|=\frac{m+n-1}{4}$ for each $i\in[4]$. For each $i\in[2]$, let $G[X_i, Y_i\uplus Y_3\uplus Y_4]$ be a complete bipartite graph. For each $i\in\{3, 4\}$, let $G[X_i, Y_1\uplus Y_2\uplus Y_i]$ be a complete bipartite graph. Color $G[X_i, Y_i]$ in blue for each $i\in[4]$, and color $G[X_1\uplus X_2, Y_3\uplus Y_4]$ and $G[X_3\uplus X_4, Y_1\uplus Y_2]$ in red. Then the red maximum connected matching has size $\frac{m+n-1}{2}<m$ since $n<m$; and the maximum blue connected matching has size $\frac{m+n-1}{4}<n$ since $m<3n$.
\end{const}

Combining Theorem \ref{1} and Szemer\'{e}di's Regularity Lemma, we obtain the following result for off-diagonal cycles. Note that if $\alpha_1>\alpha_2>0$ and $\alpha_1+\alpha_2=1$, then by equation (\ref{brc}), $br(C_{2\lfloor\alpha_1 N\rfloor}, C_{2\lfloor\alpha_2 N\rfloor})=N-1$.

\begin{theorem}\label{0} For every $\eta>0$, there exists a positive integer $N_0$ such that for every integer $N>N_0$ the following holds. Let $\alpha_1>\alpha_2>0$ such that $\alpha_1+\alpha_2=1$. Let $G[X, Y]$ be a balanced bipartite graph on $2(N-1)$ vertices with minimum degree $\delta(G)\geq(\frac{3}{4}+3\eta)(N-1)$. Then for each red-blue-edge-coloring of $G$, either there exist red even cycles of each length in $\{4, 6, 8, \ldots, (2-3\eta^2)\alpha_1N\}$, or there exist blue even cycles of each length in $\{4, 6, 8, \ldots, (2-3\eta^2)\alpha_2N\}$.
\end{theorem}

Throughout this paper, for a red-blue-edge-colored graph $G$, we use $G_R$ to denote the spanning subgraph induced by all red edges of $G$, and use $G_B$ to denote the spanning subgraph induced by all blue edges of $G$. For any $v\in V(G)$, let $N_{R}(v)=\{u\in V(G): uv\in E(G_R)\}$ and $N_{B}(v)=\{u\in V(G): uv\in E(G_B)\}$. %For a bipartite graph $G$, the complement $\overline{G}$ is the bipartite graph with the same bipartition of $G$ such that $xy\in E(\overline{G})$ if and only if $xy\notin E(G)$.
We also give a stability result for connected matchings. Before starting the result, we introduce two special colorings.

\begin{defi}\label{spe-1}
Let $0<\gamma<\frac{1}{4}$, $n<m<3n$ and $G[X, Y]$ be a balanced bipartite graph on order $2(m+n-1)$ with minimum degree $\delta(G)>(\frac{3}{4}+\gamma)(m+n-1)$. We say that a red-blue-edge-coloring of $G$ is $\gamma$-missing if there exist a partition $\{X_1, X_2\}$ of $X$ and a partition $\{Y', Y_1, Y_2\}$ of $Y$ such that

{\rm(i)} $|X_1|>m-\gamma n-1$, $|Y_1|>m-\gamma n-1$ and $|Y_2|>n-\gamma m-1$;

{\rm (ii)} For any $x\in X_1$, $N_R(x)\subseteq Y'\uplus Y_1$ and $N_B(x)\subseteq Y'\uplus Y_2$;

{\rm (iii)} For any $x\in X_2$, $N_R(x)\subseteq Y_2$ and $N_B(x)\subseteq Y_1$.
\end{defi}

\begin{defi}\label{spe-2}
Let $0<\gamma<\frac{1}{4}$, $n<m<3n$ and $G[X, Y]$ be a balanced bipartite graph on order $2(m+n-1)$ with minimum degree $\delta(G)>(\frac{3}{4}+\gamma)(m+n-1)$. We say that a red-blue-edge-coloring of $G$ is a $\gamma$-coloring if there exist $X'\subseteq X$ and a partition $\{Y', Y_1, Y_2\}$ of $Y$ such that

{\rm (i)} $|X'|>\frac{3}{4}(m+n-1)-\gamma-1$, $|Y_1|>m-\gamma n-1$ and $|Y_2|>n-\gamma m-1$;

{\rm (ii)} For any $x\in X'$, $N_R(x)\subseteq Y'\uplus Y_1$ and $N_B(x)\subseteq Y'\uplus Y_2$.
\end{defi}

We show the following stability result.

\begin{theorem}\label{Sta}
For any $0<\gamma<\frac{1}{4}$, there exists an integer $n_0>0$ such that the following holds. Let $n_0\leq n<m<3n$ and $G[V_1, V_2]$ be a balanced bipartite graph on $2(m+n-1)$ vertices with minimum degree $\delta(G)>(\frac{3}{4}+\gamma)(m+n-1)$. For each red-blue-edge-coloring of $G$ which is not $\gamma$-missing,  there exists a red connected matching of size $(1+\gamma)m$ or a blue connected matching of size $(1+\gamma)n$; or the edge coloring is a $\gamma$-coloring.
\end{theorem}

The organization of this paper goes as follows. In section 2, we show the existence of large monochromatic components. Based on this result, we will prove Theorem \ref{1} and the stability result Theorem \ref{Sta} implying the existence of large monochromatic connected matchings in section $3$. In Section 4, we use Szemer\'{e}di's Regularity Lemma to expand monochromatic connected matchings into monochromatic cycles.

\section{Monochromatic components}

In this section, we show the existence of large monochromatic components. Throughout this paper, we say a graph $G=\emptyset$ if $E(G)=\emptyset$.

\begin{fact}\label{pr1}
Let $0\leq\gamma<\frac{1}{4}$ and let $G[V_1, V_2]$ be a balanced bipartite graph on $2N$ vertices with minimum degree $\delta(G)>(\frac{3}{4}+\gamma)N$. Suppose that $G[X_1, X_2]=\emptyset$ for $X_1\subset V_1$ and $X_2\subset V_2$. If $X_j\neq\emptyset$ for some $j\in[2]$, then $|X_{3-j}|<(\frac{1}{4}-\gamma)N$.
\end{fact}
\begin{proof}
Suppose that $X_j\neq\emptyset$ for some $j\in[2]$. Since $G[X_1, X_2]=\emptyset$, $X_{3-j}\subseteq V_{3-j}\backslash N_G(x)$ for any $x\in X_j$. So $|X_{3-j}|\leq|V_{3-j}\backslash N_G(x)|\leq|V_{3-j}|-\delta(G)<(\frac{1}{4}-\gamma)N$ since $\delta(G)>(\frac{3}{4}+\gamma)N$.
\end{proof}

\begin{fact}\label{pr2}
Let $0\leq\gamma<\frac{1}{4}$ and let $G[V_1, V_2]$ be a balanced bipartite graph on $2N$ vertices with minimum degree $\delta(G)>(\frac{3}{4}+\gamma)N$. Suppose that $X_1\subset V_1$ and $X_2\subset V_2$. The following holds.

\noindent {\rm(i)} If $G[X_1, X_2]\subseteq G_B$ and $|X_i|\geq(\frac{1}{2}-2\gamma)N$ for some $i\in[2]$, then $X_{3-i}$ is contained in a blue component of $G$.

\noindent {\rm(ii)} If $G[X_1, X_2]\subseteq G_R$ and $|X_i|\geq(\frac{1}{2}-2\gamma)N$ for some $i\in[2]$, then $X_{3-i}$ is contained in a red component of $G$.
\end{fact}
\begin{proof}
(i) Recall that $\delta(G)>(\frac{3}{4}+\gamma)N$. Since $G[X_1, X_2]\subseteq G_B$, for any $x\in X_{3-i}$, $|N_B(x)\cap X_i|\geq\delta(G)-|V_i\backslash X_i|>|X_i|-(\frac{1}{4}-\gamma)N$.
For any pair vertices $x$, $x'\in X_{3-i}$, by the inclusion-exclusion principle, we have that
$$\begin{aligned}
|N_B(x)\cap N_B(x')|&\geq|N_B(x)\cap X_i|+|N_B(x')\cap X_i|-|X_i| \\
&>2(|X_i|-(\frac{1}{4}-\gamma)N)-|X_i|=|X_i|-(\frac{1}{2}-2\gamma)N\geq0
\end{aligned}$$
since $|X_i|\geq(\frac{1}{2}-2\gamma)N$. Thus $X_{3-i}$ is contained in a blue component of $G$.% say $\mathcal{B}$. Let $x\in X_{3-i}$, $|\mathcal{B}\cap X_i|\geq|N_B(x)\cap X_i|>|X_i|-(\frac{1}{4}-\gamma)N$.

(ii) The proof is similar to (i).%Now $G[X_1, X_2]\subset G_R$. For any $x\in X_{3-i}$, $|N_R(x)\cap X_i|\geq\delta(G)-|V_i\backslash X_i|>|X_i|-(\frac{1}{4}-\gamma)N$ since $\delta(G)>(\frac{3}{4}+\gamma)N$.
%For any pair vertices $x$, $x'\in X_{3-i}$, by the inclusion-exclusion principle, we have that
%$$\begin{aligned}
%|N_R(x)\cap N_R(x')|&\geq|N_R(x)\cap X_i|+|N_R(x')\cap X_i|-|X_i| \\
%&>2(|X_i|-(\frac{1}{4}-\gamma)N)-|X_i|=|X_i|-(\frac{1}{2}-2\gamma)N\geq0
%\end{aligned}$$
%since $|X_i|\geq(\frac{1}{2}-2\gamma)N$. Thus $X_{3-i}$ is contained in a red component of $G$.
\end{proof}

\begin{lemma}\label{2} For any $0\leq\gamma<\frac{1}{4}$, there exists an integer $n_0>0$ such that for any $m>n\geq n_0$ the following holds. Let $m<3n$ if $\gamma\neq0$. Let $G[V_1, V_2]$ be a balanced bipartite graph on $2(m+n-1)$ vertices with minimum degree $\delta(G)>(\frac{3}{4}+\gamma)(m+n-1)$. A red-blue-edge-coloring of $G$ yields either a red component with at least $(1+\gamma)m$ vertices in each of $V_1$ and $V_2$, or a blue component with at least $(1+\gamma)n$ vertices in each of $V_1$ and $V_2$ if and only if this edge coloring is not $\gamma$-missing.
\end{lemma}

\begin{proof} By Definition \ref{spe-1}, each $\gamma$-missing edge-coloring of $G$ yields neither a red component with at least $(1+\gamma)m$ vertices in each of $V_1$ and $V_2$; nor a blue component with at least $(1+\gamma)n$ vertices in each of $V_1$ and $V_2$.

When $\gamma=0$, we set $n_0=1$. When $\gamma>0$, we set $n_0=1+\frac{1}{\gamma}$.
Suppose that some red-blue-edge-coloring of $G$ yields neither a red component with at least $(1+\gamma)m$ vertices in each of $V_1$ and $V_2$; nor a blue component with at least $(1+\gamma)n$ vertices in each of $V_1$ and $V_2$. Then we show that this edge coloring is $\gamma$-missing.

Let $\mathcal{B}$ and $\mathcal{R}$ be a largest blue component and a largest red component of $G$ respectively. For each $i\in[2]$, let $B_i=\mathcal{B}\cap V_i$, $R_i=\mathcal{R}\cap V_i$, $BR_i=B_i\cap R_i$, and $V'_i=V_i\backslash(B_i\cup R_i)$.
By the hypothesis,
\begin{equation}\label{c0}
\min\{|B_1|, |B_2|\}<(1+\gamma)n,
\end{equation}
and
\begin{equation}\label{c1}
\min\{|R_1|, |R_2|\}<(1+\gamma)m.
\end{equation}
Let $N:=m+n-1=|V_1|=|V_2|$, now $\delta(G)>(\frac{3}{4}+\gamma)(m+n-1)=(\frac{3}{4}+\gamma)N$. When $\gamma>0$, we have that $3n>m>n\geq1+\frac{1}{\gamma}$, then
\begin{equation}\label{cn0}
(\frac{3}{4}+\gamma)N>\max\{m-\gamma n-1, (1+\gamma)m\}.
\end{equation}

\begin{claim}\label{a1} For each $i\in[2]$, the following holds.

{\rm (i)} If $x\in BR_i$, then $N_B(x)\subseteq B_{3-i}$ and $N_R(x)\subseteq R_{3-i}$.

{\rm (ii)} If $x\in B_i\backslash R_i$, then $N_B(x)\subseteq B_{3-i}$ and $N_R(x)\subseteq V'_{3-i}\uplus(B_{3-i}\backslash R_{3-i})$.

{\rm (iii)} If $x\in R_i\backslash B_i$, then $N_B(x)\subseteq V'_{3-i}\uplus(R_{3-i}\backslash B_{3-i})$ and $N_R(x)\subseteq R_{3-i}$.

{\rm (iv)} If $x\in V'_i$, then $N_B(x)\subseteq V'_{3-i}\uplus(R_{3-i}\backslash B_{3-i})$ and $N_R(x)\subseteq V'_{3-i}\uplus(B_{3-i}\backslash R_{3-i})$.

{\rm (v)} $G[BR_i, V'_{3-i}]=\emptyset$.

{\rm (vi)} $G[B_i\backslash R_i, R_{3-i}\backslash B_{3-i}]=\emptyset$.

{\rm (vii)} $B_i\neq\emptyset$ and $R_i\neq\emptyset$.
\end{claim}

\begin{proof} (i) Let $x\in BR_i=B_i\cap R_i$. Since $\mathcal{B}$ is the blue component containing  $B_i$, $N_B(x)\subseteq B_{3-i}$. Since $\mathcal{R}$ is the red component containing  $R_i$, $N_R(x)\subseteq R_{3-i}$.

(ii) Let $x\in B_i\backslash R_i$. Since $\mathcal{B}$ is the blue component containing  $B_i$, $N_B(x)\subseteq B_{3-i}$. If $N_R(x)\cap R_{3-i}\neq\emptyset$, then since $\mathcal{R}$ is a largest red component, $x\in R_i$, a contradiction. Thus $N_R(x)\subseteq V_{3-i}\backslash R_{3-i}=V'_{3-i}\uplus(B_{3-i}\backslash R_{3-i})$.

(iii) The proof is similar to (ii).

(iv) Let $x\in V'_i$. If $N_B(x)\cap B_{3-i}\neq\emptyset$, then since $\mathcal{B}$ is a largest blue component, $x\in B_i$, a contradiction. Thus $N_B(x)\subseteq V_{3-i}\backslash B_{3-i}=V'_{3-i}\uplus(R_{3-i}\backslash B_{3-i})$. If $N_R(x)\cap R_{3-i}\neq\emptyset$, then since $\mathcal{R}$ is a largest red component, $x\in R_i$, a contradiction. Thus $N_R(x)\subseteq V_{3-i}\backslash R_{3-i}=V'_{3-i}\uplus(B_{3-i}\backslash R_{3-i})$.

(v) If $BR_i=\emptyset$ or $V'_{3-i}=\emptyset$, then it is done. Suppose that $BR_i\neq\emptyset$ and $V'_{3-i}\neq\emptyset$. For any $x\in BR_i$, by (i), $N_G(x)\subseteq B_{3-i}\cup R_{3-i}$, so $N_G(x)\cap V'_{3-i}=\emptyset$. Thus $G[BR_i, V'_{3-i}]=\emptyset$.

(vi) If $B_i\backslash R_i=\emptyset$ or $R_{3-i}\backslash B_{3-i}=\emptyset$, then it is done. Suppose that $B_i\backslash R_i\neq\emptyset$ and $R_{3-i}\backslash B_{3-i}\neq\emptyset$. For any $x\in B_i\backslash R_i$, by (ii), $N_G(x)\in B_{3-i}\uplus V'_{3-i}$, so $N_G(x)\cap(R_{3-i}\backslash B_{3-i})=\emptyset$. Thus $G[B_i\backslash R_i$, $R_{3-i}\backslash B_{3-i}]=\emptyset$.

(vii) For any pair of vertices $x, x'\in V_1$, by the inclusion-exclusion principle, we have that $|N_G(x)\cap N_G(x')|\geq|N_G(x)|+|N_G(x')|-|V_2|\geq2\delta(G)-|V_2|>(\frac{1}{2}+2\gamma)N$ since $\delta(G)>(\frac{3}{4}+\gamma)N$. Thus $G$ is connected. If $B_i=\emptyset$ for some $i\in[2]$, then $G[V_1, V_2]$ is a red connected graph, contradicting to (\ref{c0}). Thus $B_i\neq\emptyset$ for each $i\in[2]$. Similarly, we have that $R_i\neq\emptyset$ for each $i\in[2]$.
\end{proof}

Next we split our argument into three cases.

{\bf Case 1.} $BR_1=BR_2=\emptyset$.

By Claim \ref{a1}(vi), we have that $G[B_1, R_2]=G[R_1, B_2]=\emptyset$. For each $i\in[2]$, by Fact \ref{pr1} and Claim \ref{a1}(vii), $|B_i|<(\frac{1}{4}-\gamma)N$ and $|R_i|<(\frac{1}{4}-\gamma)N$.

Now $|\mathcal{R}|=|R_1|+|R_2|<(\frac{1}{2}-2\gamma)N$ and  $|V'_2|=|V_2\backslash(B_2\cup R_2)|>(\frac{1}{2}+2\gamma)N$. By Claim \ref{a1}(ii), $G[B_1, V'_2]\subseteq G_R$. By Fact \ref{pr2}(ii), $B_1$ is contained in a red component of $G$, say $\mathcal{F}_1$. Let $x\in B_1=B_1\backslash R_1$, by Claim \ref{a1}(ii), $N_B(x)\subseteq B_2$ and $N_R(x)\subseteq B_2\uplus V'_2$ since $BR_2=\emptyset$, then $|\mathcal{F}_1\cap V_2|\geq|N_R(x)\cap V'_2|\geq\delta(G)-|B_2|>(\frac{1}{2}+2\gamma)N\geq
(\frac{1}{2}-2\gamma)N>|\mathcal{R}|$, a contradiction to the maximality of $\mathcal{R}$.

{\bf Case 2.} $BR_1\neq\emptyset$ and $BR_2\neq\emptyset$.

For each $i\in[2]$, by Claim \ref{a1}(i),
\begin{equation}\label{2c1}
|B_i\cup R_i|\geq\delta(G)>(\frac{3}{4}+\gamma)N,
\end{equation}
then
\begin{equation}\label{2c2}
|V'_i|=|V_i|-|B_i\cup R_i|<(\frac{1}{4}-\gamma)N.
\end{equation}

{\bf Subcase~2.1.} For some $i\in[2]$, $B_i=R_i=BR_i$.

Without loss of generality, assume that $B_1=R_1=BR_1$. Now $V_1=BR_1\uplus V'_1$. By inequality (\ref{2c1}), we have that
\begin{equation}\label{2c3}
|BR_1|=|B_1|=|R_1|>(\frac{3}{4}+\gamma)N.
\end{equation}
Then $|B_1|\overset{(\ref{2c3})}{>}(\frac{3}{4}+\gamma)N\geq(\frac{3}{2}+2\gamma)n$ since $m\geq n+1$. By inequality (\ref{c0}),
\begin{equation}\label{2c5}
|B_2|<(1+\gamma)n.
\end{equation}
If $V'_2\neq\emptyset$, then by Claim \ref{a1}(v) and Fact \ref{pr1}, $|BR_1|<(\frac{1}{4}-\gamma)N$, a contradiction to inequality (\ref{2c3}). Thus $V'_2=\emptyset$, now
\begin{equation}\label{ad1}
V_2=B_2\cup R_2.
\end{equation}

Suppose that $\gamma=0$. Now $|R_2\backslash B_2|\overset{(\ref{ad1})}{=}|V_2\backslash B_2|\overset{(\ref{2c5})}{\geq}m\geq\frac{N}{2}$ since $m\geq n+1$.
By Claim \ref{a1}(iv), $G[V'_1, R_2\backslash B_2]\subseteq G_B$. By Fact \ref{pr2}(i), $V'_1$ is contained in a blue component of $G$, say $\mathcal{F}_2$.
Since $|R_2\backslash B_2|\geq m$, $|BR_1|=|R_1|\leq m-1$ by inequality (\ref{c1}). Combining inequality (\ref{2c3}), we have that $m\geq3n+2$. Recall that $V_1=BR_1\uplus V'_1$, then
$|\mathcal{F}_2\cap V_1|\geq|V'_1|=|V_1\backslash BR_1|\geq n$.
Let $x\in V'_1$, by Claim \ref{a1}(iv) and (\ref{ad1}), $N_B(x)\subseteq R_2\backslash B_2$ and $N_R(x)\subseteq B_2\backslash R_2$, then
$|\mathcal{F}_2\cap V_2|\geq|N_B(x)\cap(R_2\backslash B_2)|\geq\delta(G)-|B_2|
\overset{(\ref{2c5})}{>}\frac{3m-n+1}{4}>2n$ since $m\geq 3n+2$.
Now $\mathcal{F}_2$ is a blue component such that $|\mathcal{F}_2\cap V_i|\geq n$ for each $i\in[2]$, a contradiction to the hypothesis.

Suppose that $\gamma\neq0$. Then $|BR_1|=|B_1|=|R_1|\overset{(\ref{2c3})}{>}(\frac{3}{4}+\gamma)N
\overset{(\ref{cn0})}{>}\max\{m-\gamma n-1, (1+\gamma)m\}>(1+\gamma)n$ since $m\geq n+1$. By inequality (\ref{c0}), $|B_2|<(1+\gamma)n$. By inequality (\ref{c1}), $|R_2|<(1+\gamma)m$. Then $|R_2\backslash B_2|\overset{(\ref{ad1})}{=}|V_2\backslash B_2|>m-\gamma n-1$ and $|B_2\backslash R_2|\overset{(\ref{ad1})}{=}|V_2\backslash R_2|>n-\gamma m-1$. Combining with Claim \ref{a1}(i) and (iv), the coloring is $\gamma$-missing as witnessed by the partitions $\{BR_1, V'_1\}$ of $V_1$ and $\{B_2\backslash R_2, BR_2, R_2\backslash B_2\}$ of $V_2$.

{\bf Subcase~2.2.} For each $i\in[2]$, $B_i\neq R_i$.

\begin{claim}\label{5}
There exists some $i\in[2]$ such that $B_i\backslash R_i=\emptyset$ and $R_i\backslash B_i=\emptyset$.
\end{claim}
\begin{proof} On the contrary, suppose that for each $i\in[2]$, either $B_i\varsubsetneq R_i$ or $R_i\varsubsetneq B_i$.

Suppose that for some $i\in[2]$, $B_i\varsubsetneq R_i$ and $R_{3-i}\varsubsetneq B_{3-i}$. Without loss of generality, assume that $B_1\varsubsetneq R_1$ and $R_2\varsubsetneq B_2$. Now $|R_1|=|B_1\cup R_1|\overset{(\ref{2c1})}{>}(\frac{3}{4}+\gamma)N$. Since $B_2\backslash R_2\neq\emptyset$, by Claim \ref{a1}(vi) and Fact \ref{pr1}, $|R_1\backslash B_1|<(\frac{1}{4}-\gamma)N$. Then $|B_1|=|R_1|-|R_1\backslash B_1|>(\frac{1}{2}+2\gamma)N\geq(1+4\gamma)n$ and
$|B_2|=|B_2\cup R_2|\overset{(\ref{2c1})}{>}(\frac{3}{4}+\gamma)N\geq(\frac{3}{2}+2\gamma)n$ since $m\geq n+1$, contradicting to $(\ref{c0})$.

Suppose that $R_i\varsubsetneq B_i$ for each $i\in[2]$, then each $|B_i|\overset{(\ref{2c1})}{>}(\frac{3}{4}+\gamma)N\geq(\frac{3}{2}+2\gamma)n$ since $m\geq n+1$, contradicting to (\ref{c0}).

Suppose that $B_i\varsubsetneq R_i$ for each $i\in[2]$, then each $|R_i|\overset{(\ref{2c1})}{>}(\frac{3}{4}+\gamma)N$. If $\gamma\neq0$, then $|R_i|>(\frac{3}{4}+\gamma)N\overset{(\ref{cn0})}{>}(1+\gamma)m$ for each $i\in[2]$, contradicting to (\ref{c1}). Next we consider the case when $\gamma=0$. By inequality (\ref{c1}), we can assume that $|R_1|\leq m-1$. Since $B_1\varsubsetneq R_1$, $|V'_1|=|V_1\backslash R_1|\geq n$. By Claim \ref{a1}(v) and Fact \ref{pr1}, $|B_2|<\frac{N}{4}$, then $|R_2\backslash B_2|\geq|R_2|-\frac{N}{4}\geq\frac{N}{2}$. By Claim \ref{a1}(iv), $G[V'_1, R_2\backslash B_2]\subseteq G_B$. By Fact \ref{pr2}(i), $V'_1$ is contained in a blue component of $G$, say $\mathcal{F}_3$. Now $|\mathcal{F}_3\cap V_1|\geq|V'_1|\geq n$. Let $x\in V'_1$, by Claim \ref{a1}(iv), $N_B(x)\subseteq V'_2\uplus(R_2\backslash B_2)$ and $N_R(x)\subseteq V'_2$ since $B_2\varsubsetneq R_2$, then
$|\mathcal{F}_3\cap V_2|\geq|N_B(x)\cap(R_2\backslash B_2)|\geq\delta(G)-|V'_2|\overset{(\ref{2c2})}{>}\frac{m+n-1}{2}\geq n$ since $m\geq n+1$. Now $\mathcal{F}_3$ is a blue component such that $|\mathcal{F}_3\cap V_i|\geq n$ for each $i\in[2]$, a contradiction to the hypothesis.

This completes the proof of Claim \ref{5}.
\end{proof}

By Claim \ref{5}, we can assume that $B_1\backslash R_1\neq\emptyset$ and $R_1\backslash B_1\neq\emptyset$. By Claim \ref{a1}(vi) and Fact \ref{pr1}, $|B_2\backslash R_2|<(\frac{1}{4}-\gamma)N$ and $|R_2\backslash B_2|<(\frac{1}{4}-\gamma)N$.

Now $|BR_2|=|B_2\cup R_2|-|R_2\backslash B_2|-|B_2\backslash R_2|\overset{(\ref{2c1})}{>}(\frac{1}{4}+3\gamma)N$.
Then $V'_1=\emptyset$, otherwise by Claim \ref{a1}(v) and Fact \ref{pr1}, $|BR_2|<(\frac{1}{4}-\gamma)N$, a contradiction. Thus
\begin{equation}\label{ad2}
V_1=B_1\cup R_1.
\end{equation}
Note that $|B_2|=|B_2\cup R_2|-|R_2\backslash B_2|\overset{(\ref{2c1})}{>}(\frac{1}{2}+2\gamma)N\geq(1+4\gamma)n$ since $m\geq n+1$.
By inequality (\ref{c0}), we have that
\begin{equation}\label{2c9}
|B_1|<(1+\gamma)n.
\end{equation}
Then $|R_1\backslash B_1|\overset{(\ref{ad2})}{=}|V_1\backslash B_1|>m-\gamma n-1>(\frac{1}{4}-\gamma)N$ since $m\geq n+1$. Now $B_2\backslash R_2=\emptyset$, otherwise by Claim \ref{a1}(vi) and Fact \ref{pr1}, $|R_1\backslash B_1|<(\frac{1}{4}-\gamma)N$, a contradiction. By the assumption $B_2\neq R_2$, we have that
\begin{equation}\label{ad3}
B_2\varsubsetneq R_2,
\end{equation}
and so
\begin{equation}\label{2c11}
|R_2|=|B_2\cup R_2|\overset{(\ref{2c1})}{>}(\frac{3}{4}+\gamma)N.
\end{equation}

Suppose that $\gamma\neq0$, now $3n>m>n\geq1+\frac{1}{\gamma}$. Then $|R_2|\overset{(\ref{2c11})}{>}(\frac{3}{4}+\gamma)N\overset{(\ref{cn0})}{>}(1+\gamma)m$.
By inequality (\ref{c1}), $|R_1|<(1+\gamma)m$, then $|B_1\backslash R_1|\overset{(\ref{ad2})}{=}|V_1\backslash R_1|>n-\gamma m-1>(\frac{1}{4}-\gamma)N$ since $3n>m>1+\frac{1}{\gamma}$. If $R_2\backslash B_2\neq\emptyset$, by Claim \ref{a1}(vi) and Fact \ref{pr1}, now $|B_1\backslash R_1|<(\frac{1}{4}-\gamma)N$, a contradiction. Thus $R_2\backslash B_2=\emptyset$, contradicting to (\ref{ad3}).

Suppose that $\gamma=0$. Now $|R_1\backslash B_1|\overset{(\ref{ad2})}{=}|V_1\backslash B_1|\overset{(\ref{2c9})}{\geq}m\geq\frac{N}{2}$ since $m\geq n+1$. By Claim \ref{a1}(iv), $G[R_1\backslash B_1, V'_2]\subseteq G_B$. By Fact \ref{pr2}(i), $V'_2$ is contained in a blue component of $G$, say $\mathcal{F}_4$. Since $|R_1\backslash B_1|\geq m$, by inequality (\ref{c1}), $|R_2|\leq m-1$. Combining with inequality (\ref{2c11}), we have that $m\geq 3n+2$. By (\ref{ad3}), $|\mathcal{F}_4\cap V_2|\geq|V'_2|=|V_2\backslash R_2|\geq n$. Let $x\in V'_2$, by Claim \ref{a1}(iv) and (\ref{ad2}), $N_B(x)\subseteq R_1\backslash B_1$ and $N_R(x)\subseteq B_1\backslash R_1$, then
$|\mathcal{F}_4\cap V_1|\geq|N_B(x)\cap(R_1\backslash B_1)|\geq\delta(G)-|B_1|\overset{(\ref{2c9})}{>}\frac{3m-n+1}{4}>2n$ since $m\geq3n+2$.
Now $\mathcal{F}_4$ is a blue component such that $|\mathcal{F}_4\cap V_i|\geq n$ for each $i\in[2]$, a contradiction to the hypothesis.

{\bf Case 3.} Exactly one of $BR_1$ and $BR_2$ is empty.

Without loss of generality, assume that $BR_1=\emptyset$ and $BR_2\neq\emptyset$. Now $V_1=B_1\uplus R_1\uplus V'_1$. For any $x\in BR_2$, by Claim \ref{a1}(i), $N_G(x)\subseteq B_1\uplus R_1$. Then
\begin{equation}\label{3c1}
|B_1|+|R_1|\geq\delta(G)>(\frac{3}{4}+\gamma)N,
\end{equation}
and so
\begin{equation}\label{3c2}
|V'_1|=|V_1\backslash(B_1\uplus R_1)|<(\frac{1}{4}-\gamma)N.
\end{equation}
By averaging principle, either $|B_1|>(\frac{3}{8}+\frac{\gamma}{2})N$ or $|R_1|>(\frac{3}{8}+\frac{\gamma}{2})N$. Suppose that $R_2\backslash B_2\neq\emptyset$ and $B_2\backslash R_2\neq\emptyset$. By Claim \ref{a1}(vi) and Fact \ref{pr1}, since $BR_1=\emptyset$, $|B_1|=|B_1\backslash R_1|<(\frac{1}{4}-\gamma)N$ and $|R_1|=|R_1\backslash B_1|<(\frac{1}{4}-\gamma)N$, a contradiction. Thus either $R_2\subseteq B_2$ or $B_2\subseteq R_2$.

Now we only need to consider the following possibilities.

{\bf Subcase~3.1.} $BR_1=\emptyset$ and $R_2\varsubsetneq B_2$.

Since $BR_1=\emptyset$, $G[R_1, B_2\backslash R_2]=\emptyset$ by Claim \ref{a1}(vi). Since $B_2\backslash R_2\neq\emptyset$, by Fact \ref{pr1},
\begin{equation}\label{3c21}
|R_1|<(\frac{1}{4}-\gamma)N.
\end{equation}
Then
\begin{equation}\label{3c22}
|B_1|\overset{(\ref{3c1})}{\geq}\delta(G)-|R_1|>(\frac{1}{2}+2\gamma)N.
\end{equation}
Now $|B_1|\overset{(\ref{3c22})}{>}(\frac{1}{2}+2\gamma)N\geq(1+4\gamma)n$ since $m\geq n+1$. By inequality (\ref{c0}), $|B_2|<(1+\gamma)n$. Since $R_2\varsubsetneq B_2$, $|V'_2|=|V_2\backslash B_2|>m-\gamma n-1\geq(\frac{1}{2}-2\gamma)N$ since $m\geq n+1$. Since $BR_1=\emptyset$, $G[B_1, V'_2]\subseteq G_R$ by Claim \ref{a1}(ii). By Fact \ref{pr2}(ii), $B_1$ is contained in a red component of $G$, say $\mathcal{H}_1$.

For any $x\in V'_2$, by Claim \ref{a1}(iv), $N_B(x)\subseteq R_1\uplus V'_1$ and $N_R(x)\subseteq B_1\uplus V'_1$ since $BR_1=\emptyset$, then $|N_R(x)\cap B_1|\geq\delta(G)-|R_1\uplus V'_1|
\overset{(\ref{3c2})}{\underset{(\ref{3c21})}{>}}(\frac{1}{4}+3\gamma)N$.
Since $B_1\subseteq\mathcal{H}_1\cap V_1$, $V'_2\subseteq\mathcal{H}_1\cap V_2$. Recall that  $R_2\varsubsetneq B_2$ and $|B_2|<(1+\gamma)n$. Then $|\mathcal{R}|=|R_1|+|R_2|\overset{(\ref{3c21})}{<}(\frac{1}{4}-\gamma)N+|B_2|
<(\frac{1}{4}-\gamma)(m-1)+\frac{5n}{4}$ and
$|\mathcal{H}_1|\geq|B_1|+|V'_2|\overset{(\ref{3c22})}{>}(\frac{1}{2}+2\gamma)N+|V'_2|
=(\frac{1}{2}+2\gamma)N+|V_2\backslash B_2|>(\frac{3}{2}+2\gamma)(m-1)+(\frac{1}{2}+\gamma)n$.
Since $m\geq n+1$, $|\mathcal{H}_1|>(\frac{3}{2}+2\gamma)(m-1)+(\frac{1}{2}+\gamma)n
>(\frac{1}{4}-\gamma)(m-1)+\frac{5}{4}n>|\mathcal{R}|$, contradicting to the maximality of $\mathcal{R}$.

{\bf Subcase~3.2.} $BR_1=\emptyset$ and $B_2\varsubsetneq R_2$.

Since $BR_1=\emptyset$, $G[B_1, R_2\backslash B_2]=\emptyset$ by Claim \ref{a1}(vi). Since $R_2\backslash B_2\neq\emptyset$, by Claim \ref{a1}(vii) and Fact \ref{pr1}, we have that
\begin{equation}\label{3c31}
|R_2\backslash B_2|<(\frac{1}{4}-\gamma)N,
\end{equation}
and
\begin{equation}\label{3c32}
|B_1|<(\frac{1}{4}-\gamma)N.
\end{equation}
Then
\begin{equation}\label{3c33}
|R_1|\overset{(\ref{3c1})}{\geq}\delta(G)-|B_1|>(\frac{1}{2}+2\gamma)N.
\end{equation}
Recall that $BR_1=\emptyset$. For any $x\in R_1$, by Claim \ref{a1}(iii), $N_B(x)\subseteq V'_2\uplus(R_2\backslash B_2)$ and $N_R(x)\subseteq R_2$, then
\begin{equation}\label{3c68}
|N_B(x)\cap V'_2|\geq\delta(G)-|R_2|.
\end{equation}

\begin{claim}\label{n03}
The following holds.

{\rm(i)} $V'_2\neq\emptyset$. Besides, $V'_2$ is contained in some blue component of $G$, say $\mathcal{H}_2$.

{\rm(ii)} $V'_1=\emptyset$. Furthermore, $V_1=B_1\uplus R_1$.

{\rm(iii)} $|V'_2|<(1+\gamma)n$.
\end{claim}
\begin{proof}(i) Suppose that $V'_2=\emptyset$. Recall that $B_2\varsubsetneq R_2$, then $R_2=V_2$ and
$|BR_2|=|B_2|=|R_2|-|R_2\backslash B_2|=|V_2|-|R_2\backslash B_2|\overset{(\ref{3c31})}{>}(\frac{3}{4}+\gamma)N$.
We have that  $V'_1=\emptyset$, otherwise by Claim \ref{a1}(v) and Fact \ref{pr1}, $|BR_2|<(\frac{1}{4}-\gamma)N$, a contradiction.
Note that $|B_2|>(\frac{3}{4}+\gamma)N\geq(\frac{3}{2}+2\gamma)n$ since $m\geq n+1$. By inequality (\ref{c0}), $|B_1|<(1+\gamma)n$. Since $R_2=V_2$, by inequality (\ref{c1}), $|R_1|<(1+\gamma)m$. If $\gamma=0$, now $|B_1|+|R_1|\leq m+n-2$, then $|V'_1|=|V_1\backslash(B_1\cup R_1)|\geq1$, a contradiction. If $\gamma\neq0$,  then $|V'_1|=|V_1\backslash(B_1\cup R_1)|\overset{(\ref{3c32})}{>}(\frac{3}{4}+\gamma)N-|R_1|
>(\frac{3}{4}+\gamma)N-(1+\gamma)m\overset{(\ref{cn0})}{>}0$, a contradiction. Thus $V'_2\neq\emptyset$.

Since $BR_1=\emptyset$,  by Claim \ref{a1}(iii), $G[R_1, V'_2]\subseteq G_B$. Combining Fact \ref{pr2}(i) and inequality (\ref{3c33}), $V'_2$ is contained in a blue component of $G$, say $\mathcal{H}_2$.

(ii) Suppose that $V'_1\neq\emptyset$. Recall that $B_2\varsubsetneq R_2$. By Claim \ref{a1}(v) and Fact \ref{pr1}, $|B_2|=|BR_2|<(\frac{1}{4}-\gamma)N$, then
$|R_2|=|B_2|+|R_2\backslash B_2|\overset{(\ref{3c31})}{<}(\frac{1}{2}-2\gamma)N$.
Let $x\in R_1$, $|N_B(x)\cap V'_2|\overset{(\ref{3c68})}{\geq}\delta(G)-|R_2|>(\frac{1}{4}+3\gamma)N$. By (i), $R_1\subseteq\mathcal{H}_2\cap V_1$. Since $m\geq n+1$, $|\mathcal{H}_2\cap V_2|\geq|V'_2|=|V_2\backslash R_2|>(\frac{1}{2}+2\gamma)N\geq(1+4\gamma)n$ and $|\mathcal{H}_2\cap V_1|\geq|R_1|\overset{(\ref{3c33})}{>}(\frac{1}{2}+2\gamma)N\geq(1+4\gamma)n$, a contradiction to the hypothesis. Thus $V'_1=\emptyset$. Since $BR_1=\emptyset$, $V_1=B_1\uplus R_1$.

(iii) By (i), $V'_2\subseteq\mathcal{H}_2\cap V_2$.
Let $x\in V'_2$, by (ii) and Claim \ref{a1}(iv), $N_B(x)\subseteq R_1$ and $N_R(x)\subseteq B_1$, then $|\mathcal{H}_2\cap V_1|\geq|N_B(x)\cap R_1|
\geq\delta(G)-|B_1|\overset{(\ref{3c32})}{>}(\frac{1}{2}+2\gamma)N\geq(1+4\gamma)n$
since $m\geq n+1$. If $|V'_2|\geq(1+\gamma)n$, then $\mathcal{H}_2$ is a blue component such that $|\mathcal{H}_2\cap V_i|\geq(1+\gamma)n$ for each $i\in[2]$, a contradiction to the hypothesis. Thus $|V'_2|<(1+\gamma)n$.
\end{proof}

Suppose that $\gamma=0$. By Claim \ref{n03}(iii), $|V'_2|\leq n-1$. Since $B_2\varsubsetneq R_2$, $|R_2|=|V_2\backslash V'_2|\geq m$. By inequality (\ref{c1}), $|R_1|\leq m-1$. By Claim \ref{n03}(ii), $|B_1|=|V_1\backslash R_1|\geq n$. Combining with inequality (\ref{3c32}), we have that $m\geq 3n+2$.
By inequality (\ref{c0}), $|B_2|\leq n-1$. Since $B_2\varsubsetneq R_2$, $|R_2\backslash B_2|=|V_2\backslash(B_2\uplus V'_2)|>|V_2|-2(n-1)=m-n+1>\frac{m+n}{2}$ since $m\geq 3n+2$, a contradiction to inequality (\ref{3c31}).

Suppose that $\gamma\neq0$, now $3n>m>n\geq1+\frac{1}{\gamma}$. By Claim \ref{n03}(ii),
$|R_1|=|V_1\backslash B_1|\overset{(\ref{3c32})}{>}(\frac{3}{4}+\gamma)N\overset{(\ref{cn0})}{>}(1+\gamma)m$. By inequality (\ref{c1}), $|R_2|<(1+\gamma)m$. For any $x\in R_1$, $|N_B(x)\cap V'_2|\overset{(\ref{3c68})}{\geq}\delta(G)-|R_2|
>(\frac{3}{4}+\gamma)N-(1+\gamma)m\overset{(\ref{cn0})}{>}0$. By Claim \ref{n03}(i), $R_1\subseteq\mathcal{H}_2\cap V_1$. Since $B_2\varsubsetneq R_2$ and $|R_2|<(1+\gamma)m$,
$|\mathcal{H}_2|\geq|R_1|+|V'_2|>(\frac{3}{4}+\gamma)N+|V'_2|=(\frac{3}{4}+\gamma)N+|V_2\backslash R_2|
>(\frac{3}{4}+2\gamma)m+(\frac{7}{4}+\gamma)(n-1)$ and $|\mathcal{B}|=|B_1|+|B_2|\overset{(\ref{3c32})}{<}(\frac{1}{4}-\gamma)N+|R_2|
<\frac{5m}{4}+(\frac{1}{4}-\gamma)n$.
Since $3n>m>1+\frac{1}{\gamma}$, $|\mathcal{H}_2|>(\frac{3}{4}+2\gamma)m+(\frac{7}{4}+\gamma)(n-1)
>\frac{5m}{4}+(\frac{1}{4}-\gamma)n>|\mathcal{B}|$, a contradiction to the maximality of $\mathcal{B}$.

{\bf Subcase~3.3.} $BR_1=\emptyset$ and $B_2=R_2=BR_2$.

For any $x\in B_1$, by Claim \ref{a1}(ii), $N_B(x)\subseteq B_2=BR_2$ and $N_R(x)\subseteq V'_2$, then
\begin{equation}\label{3c7}
|N_R(x)\cap V'_2|\geq\delta(G)-|BR_2|.
\end{equation}
For any $x\in R_1$, by Claim \ref{a1}(iii), $N_R(x)\subseteq R_2=BR_2$ and $N_B(x)\subseteq V'_2$, then
\begin{equation}\label{3c8}
|N_B(x)\cap V'_2|\geq\delta(G)-|BR_2|.
\end{equation}
For any $x\in V'_2$, by Claim \ref{a1}(iv), $N_B(x)\subseteq R_1\uplus V'_1$ and $N_R(x)\subseteq B_1\uplus V'_1$, then
\begin{equation}\label{3c10}
|N_B(x)\cap R_1|\geq\delta(G)-|B_1\uplus V'_1|
\end{equation}
and
\begin{equation}\label{3c9}
|N_R(x)\cap B_1|\geq\delta(G)-|R_1\uplus V'_1|
\end{equation}

\begin{claim}\label{n01}
The following holds.

{\rm(i)} $|B_2|=|R_2|=|BR_2|\geq\frac{N}{2}$.

{\rm(ii)} $V'_1=\emptyset$. Furthermore, $V_1=B_1\uplus R_1$.

{\rm(iii)} $|B_1|<(1+\gamma)n$.

{\rm(iv)} $|R_1|<(1+\gamma)m$.
\end{claim}

\begin{proof} (i) Suppose that
\begin{equation}\label{3ca}
|B_2|=|R_2|=|BR_2|<\frac{N}{2}<|V_2\backslash BR_2|=|V'_2|.
\end{equation}
Since $BR_1=\emptyset$, by Claim \ref{a1}(iv), $G[B_1, V'_2]\subseteq G_R$ and $G[R_1, V'_2]\subseteq G_B$. Combining Fact \ref{pr2} and inequality (\ref{3ca}), $B_1$ is contained in a red component of $G$, say $\mathcal{H}_3$, and $R_1$ is contained in a blue component of $G$, say $\mathcal{H}_4$.

We claim that $V'_1=\emptyset$. On the contrary, suppose that $V'_1\neq\emptyset$. By Claim \ref{a1}(v) and Fact \ref{pr1}, we have that
\begin{equation}\label{nn1}
|B_2|=|R_2|=|BR_2|<(\frac{1}{4}-\gamma)N.
\end{equation}
Recall that $B_1\subseteq\mathcal{H}_3\cap V_1$. Let $x\in B_1$, then
$|\mathcal{H}_3\cap V_2|\geq|N_R(x)\cap V'_2|
\overset{(\ref{3c7})}{\geq}\delta(G)-|B_2|
\overset{(\ref{nn1})}{>}(\frac{1}{2}+2\gamma)N\overset{(\ref{3ca})}{\geq}|R_2|$.
Since $\mathcal{R}$ is a largest red component, $|\mathcal{H}_3|\leq|\mathcal{R}|$. Then
$|B_1|\leq|\mathcal{H}_3\cap V_1|=|\mathcal{H}_3|-|\mathcal{H}_3\cap V_2|<|\mathcal{R}|-|R_2|=|R_1|$.
Recall that $R_1\subseteq\mathcal{H}_4\cap V_1$. Let $x\in R_1$, then
$|\mathcal{H}_4\cap V_2|\geq|N_B(x)\cap V'_2|\overset{(\ref{3c8})}{\geq}\delta(G)-|R_2|
\overset{(\ref{nn1})}{>}(\frac{1}{2}+2\gamma)N\overset{(\ref{3ca})}{\geq}|B_2|$.
Thus $|\mathcal{H}_4|=|\mathcal{H}_4\cap V_1|+|\mathcal{H}_4\cap V_2|>|R_1|+|B_2|>|B_1|+|B_2|=|\mathcal{B}|$, a contradiction to the maximality of $\mathcal{B}$.
By the assumption $BR_1=\emptyset$, now $V_1=B_1\uplus R_1$.

Suppose that $|R_1|\leq\frac{N}{2}$, then $|B_1|=|V_1\backslash R_1|\geq\frac{N}{2}\geq|R_1|$.
For any $x\in V'_2$,
$|N_R(x)\cap B_1|\overset{(\ref{3c9})}{>}(\frac{3}{4}+\gamma)N-|R_1|\geq(\frac{1}{4}+\gamma)N$. Recall that $B_1\subseteq\mathcal{H}_3\cap V_1$, then $V'_2\subseteq\mathcal{H}_3\cap V_2$. Thus
$|\mathcal{H}_3|\geq|B_1|+|V'_2|\geq|R_1|+|V'_2|\overset{(\ref{3ca})}{>}|R_1|+|R_2|=|\mathcal{R}|$, contradicting to the maximality of $\mathcal{R}$.

Next we assume that $|R_1|>\frac{N}{2}$, then $|B_1|=|V_1\backslash R_1|<\frac{N}{2}<|R_1|$.
For any $x\in V'_2$,
$|N_B(x)\cap R_1|\overset{(\ref{3c10})}{>}(\frac{3}{4}+\gamma)N-|B_1|>(\frac{1}{4}+\gamma)N$.
Recall that $R_1\subseteq\mathcal{H}_4\cap V_1$, then $V'_2\subseteq\mathcal{H}_4\cap V_2$. Now $|\mathcal{H}_4|\geq|R_1|+|V'_2|>|B_1|+|V'_2|\overset{(\ref{3ca})}{>}|B_1|+|B_2|=|\mathcal{B}|$, contradicting to the maximality of $\mathcal{B}$.

(ii) If $V'_1\neq\emptyset$, then by Claim \ref{a1}(v) and Fact \ref{pr1},  $|BR_2|<(\frac{1}{4}-\gamma)N$, contradicting to (i). Thus $V'_1=\emptyset$. Recall that $BR_1=\emptyset$, then $V_1=B_1\uplus R_1$.

(iii) Suppose that $|B_1|\geq(1+\gamma)n$. By inequality (\ref{c0}),
\begin{equation}\label{3c3}
|BR_2|=|R_2|=|B_2|<(1+\gamma)n\leq|B_1|.
\end{equation}
Combining with (i), we have that $m-1<(1+2\gamma)n$, then $|B_1|\overset{(\ref{3c3})}{\geq}(1+\gamma)n>(\frac{1}{2}-2\gamma)N$.
Since $BR_1=\emptyset$, by Claim \ref{a1}(iv), $G[B_1, V'_2]\subseteq G_R$. By Fact \ref{pr2}(ii), $V'_2$ is contained in a red component of $G$, say $\mathcal{H}_5$.

For any $x\in B_1$, $|N_R(x)\cap V'_2|\overset{(\ref{3c7})}{\geq}\delta(G)-|BR_2|
\overset{(\ref{3c3})}{>}(\frac{3}{4}+\gamma)N-(1+\gamma)n\geq(\frac{1}{2}+\gamma)n$
since $m\geq n+1$. Since $V'_2\subseteq \mathcal{H}_5\cap V_2$, $B_1\subseteq \mathcal{H}_5\cap V_1$. By (ii), $|R_1|=|V_1\backslash B_1|\overset{(\ref{3c3})}{<}|V_2\backslash BR_2|=|V'_2|$. Then $|\mathcal{H}_5|\geq|B_1|+|V'_2|\overset{(\ref{3c3})}{>}|R_2|+|V'_2|>|R_2|+|R_1|=|\mathcal{R}|$, a contradiction to the maximality of $\mathcal{R}$. Thus $|B_1|<(1+\gamma)n$.

(iv) Suppose that $|R_1|\geq(1+\gamma)m$. By inequality (\ref{c1}),
\begin{equation}\label{3ca7}
|BR_2|=|B_2|=|R_2|<(1+\gamma)m\leq|R_1|.
\end{equation}
Now $|R_1|\overset{(\ref{3ca7})}{\geq}(1+\gamma)m>(\frac{1}{2}-2\gamma)N$ since $m\geq n+1$. Since $BR_1=\emptyset$, by Claim \ref{a1}(iv), $G[R_1, V'_2]\subseteq G_B$. By Fact \ref{pr2}(i), $V'_2$ is contained in a blue component of $G$, say $\mathcal{H}_6$.

Suppose that $\gamma\neq0$. For any $x\in R_1$, $|N_B(x)\cap V'_2|\overset{(\ref{3c8})}{\geq}\delta(G)-|BR_2|
\overset{(\ref{3ca7})}{>}(\frac{3}{4}+\gamma)N-(1+\gamma)m\overset{(\ref{cn0})}{>}0$.
Since $V'_2\subseteq \mathcal{H}_6\cap V_2$, $R_1\subseteq \mathcal{H}_6\cap V_1$.
By (ii), $|B_1|=|V_1\backslash R_1|\overset{(\ref{3ca7})}{<}|V_2\backslash BR_2|=|V'_2|$. Then $|\mathcal{H}_6|\geq|R_1|+|V'_2|\overset{(\ref{3ca7})}{>}|B_2|+|V'_2|>|B_2|+|B_1|=|\mathcal{B}|$, a contradiction to the maximality of $\mathcal{B}$.

Suppose that $\gamma=0$. Now $|\mathcal{H}_6\cap V_2|\geq|V'_2|=|V_2\backslash BR_2|\overset{(\ref{3ca7})}{\geq}n$.
Suppose that $m>3n$. Let $x\in V'_2$, by (ii), then $|\mathcal{H}_6\cap V_1|\geq|N_B(x)\cap R_1|
\overset{(\ref{3c10})}{\geq}\delta(G)-|B_2|\overset{(\ref{3ca7})}{>}\frac{3m-n+1}{4}>2n$.
Now $\mathcal{H}_6$ is a blue component such that $|\mathcal{H}_6\cap V_i|\geq n$ for each $i\in[2]$, a contradiction. Thus we have that $m\leq3n$. For any $x\in R_1$,
$|N_B(x)\cap V'_2|\overset{(\ref{3c8})}{\geq}\delta(G)-|BR_2|\overset{(\ref{3ca7})}{>}\frac{3n-m+1}{4}>0$.
Since $V'_2\subseteq \mathcal{H}_6\cap V_2$, $R_1\subseteq \mathcal{H}_6\cap V_1$. Then $\mathcal{H}_6$ is a blue component such that $|\mathcal{H}_6\cap V_1|\geq|R_1|\overset{(\ref{3ca7})}{\geq}m\geq n+1$ and $|\mathcal{H}_6\cap V_2|\geq|V'_2|\geq n$, a contradiction. Thus $|R_1|<(1+\gamma)m$.

This completes the proof of Claim \ref{n01}.
\end{proof}

If $\gamma=0$, by Claim \ref{n01}(iii)-(iv), we have that $|B_1|\leq n-1$ and $|R_1|\leq m-1$, then by Claim \ref{n01}(ii), $|V_1|=|B_1|+|R_1|\leq m+n-2$, a contradiction.

Next let $\gamma\neq0$. Recall that $BR_1=\emptyset$ and $B_2=R_2$, then except $\mathcal{R}$, all other red components of $G$ are between $B_1$ and $V'_2$. Combining inequality (\ref{c1}) and Claim \ref{n01}(iii), there is no red component of $G$ with at least $(1+\gamma)m$ vertices in $V_1$ and at least $(1+\gamma)m$ vertices in $V_2$.

By Claim \ref{n01}(ii)-(iii), $|R_1|=|V_1\backslash B_1|>m-\gamma n-1>(\frac{1}{2}-2\gamma)N$ since $m\geq n+1$. Since $BR_1=\emptyset$, by Claim \ref{a1}(iv), $G[R_1, V'_2]\subseteq G_B$. By Fact \ref{pr2}(i), $V'_2$ is contained in a blue component of $G$, say $\mathcal{C}$. For each $i\in[2]$, let $C_i=\mathcal{C}\cap V_i$. By Claim \ref{a1}(iv) and Claim \ref{n01}(ii), $C_1\subseteq R_1$. Since $B_2=R_2$, $C_2=V'_2$, thus $V_2=BR_2\uplus C_2$. By Claim \ref{n01}(ii)-(iv), $|B_1|=|V_1\backslash R_1|>n-\gamma m-1$ and $|R_1|=|V_1\backslash B_1|>m-\gamma n-1$.

Suppose that $C_1=R_1$. Suppose that $|V'_2|\geq(1+\gamma)n$, then $|R_1|=|C_1|<(1+\gamma)n$ by the hypothesis. By Claim \ref{n01}(ii)-(iii), $N=|V_1|=|B_1\uplus R_1|<2(1+\gamma)n$, then $m-1<(1+2\gamma)n$. By Claim \ref{n01}(i), $(1+\gamma)n\leq|V'_2|=|V_2\backslash BR_2|\leq\frac{N}{2}$, and so $m-1\geq(1+2\gamma)n$, a contradiction. Thus $|V'_2|<(1+\gamma)n$. Since $B_2=R_2=BR_2$, $|BR_2|=|V_2\backslash V'_2|>m-\gamma n-1$. Combining Claim \ref{a1}(ii)-(iv) and Claim \ref{n01}(ii), the coloring is $\gamma$-missing as witnessed by the partitions $\{B_1, R_1\}$ of $V_1$ and $\{BR_2, V'_2\}$ of $V_2$.

Suppose that $C_1\varsubsetneq R_1$. Recall that $BR_2=B_2=R_2$. Let $x\in R_1\backslash C_1$. Since $\mathcal{R}$ is a largest red component, $N_R(x)\subseteq R_2$. By Claim \ref{n01}(ii), $x\in V_1\backslash(B_1\uplus C_1)$, then $N_B(x)=\emptyset$ since $V_2=B_2\uplus C_2$. Then $|BR_2|=|R_2|\geq\delta(G)>(\frac{3}{4}+\gamma)N\overset{(\ref{cn0})}{>}m-\gamma n-1$. Combining Claim \ref{a1}(ii)-(iv) and Claim \ref{n01}(ii), the coloring is $\gamma$-missing as witnessed by the partitions $\{B_1, R_1\}$ of $V_1$ and $\{BR_2, V'_2\}$ of $V_2$.
\end{proof}

\section{Monochromatic connected matchings}

In this section, we will prove Theorem \ref{1} and Theorem \ref{Sta}. Let $\alpha'(G)$ denote the number of edges in a maximum matching of $G$. A {\em vertex cover} of a graph $G$ is a set $Q\subseteq V(G)$ which contains at least one endpoint of each edge of $G$. We will apply the K\"{o}nig-Egerv\'{a}ry Theorem in the proofs of Theorem \ref{1} and Theorem \ref{Sta}.

\begin{theorem}[The K\"{o}nig-Egerv\'{a}ry Theorem]\label{EK}
In any bipartite graph, the number of edges in a maximum matching is equal to the number of vertices in a minimum vertex cover.
\end{theorem}

We will prove several crucial lemmas before giving the proofs of Theorem \ref{1} and Theorem \ref{Sta}.

\begin{lemma}\label{3} For any $0\leq\gamma<\frac{1}{4}$, there exists an integer $n_0>0$ such that for any $m>n\geq n_0$ the following holds. Let $m<3n$ if $\gamma\neq0$. Let $G[V_1, V_2]$ be a balanced bipartite graph on $2(m+n-1)$ vertices with minimum degree $\delta(G)>(\frac{3}{4}+\gamma)(m+n-1)$. Suppose that for some red-blue-edge-coloring of $G$ which is not $\gamma$-missing, there exists no red connected matching of size $(1+\gamma)m$. Then either there exists a blue component of G with at least $(1+\gamma)n$ vertices in each of $V_1$ and $V_2$, or the edge coloring is a $\gamma$-coloring.
\end{lemma}
\begin{proof} When $\gamma=0$, we set $n_0=1$. When $\gamma>0$, we set $n_0=1+\frac{1}{\gamma}$. By Lemma \ref{2}, for each red-blue-edge-coloring of $G$ which is not $\gamma$-missing, there exists either a red component with at least $(1+\gamma)m$ vertices in each of $V_1$ and $V_2$; or a blue component with at least $(1+\gamma)n$ vertices in each of $V_1$ and $V_2$. In the latter case, we are done. Thus we assume that $\mathcal{R}$ is a largest red component of $G$ such that $|\mathcal{R}\cap V_i|\geq(1+\gamma)m$ for each $i\in[2]$. Let $T$ be a minimum vertex cover of $\mathcal{R}$. For each $i\in[2]$, let $R_i=\mathcal{R}\cap V_i$, $T_i=T\cap V_i$, $R'_i=R_i\backslash T_i$ and $V'_i=V_i\backslash R_i$, then $V_i=T_i\uplus R'_i\uplus V'_i$. Let $N:=m+n-1$.

\begin{claim}\label{b0} For each $i\in[2]$, the following holds.

{\rm (i)} If $x\in R'_i$, then $N_R(x)\subseteq T_{3-i}$ and $N_G(x)\cap(R'_{3-i}\uplus V'_{3-i})\subseteq N_B(x)$.

{\rm (ii)} If $x\in V'_i$, then $N_R(x)\subseteq V'_{3-i}$ and $N_G(x)\cap(R'_{3-i}\uplus T_{3-i})\subseteq N_B(x)$.
\end{claim}

\begin{proof} (i) Let $x\in R'_i=R_i\backslash T_i$. Since $T$ is a minimum vertex cover of $\mathcal{R}$, $N_R(x)\subseteq T_{3-i}$. Then $N_G(x)\cap(V_{3-i}\backslash T_{3-i})=N_G(x)\cap(R'_{3-i}\uplus V'_{3-i})\subseteq N_B(x)$.

(ii) Let $x\in V'_i=V_i\backslash R_i$. If $N_R(x)\cap R_{3-i}\neq\emptyset$, since $\mathcal{R}$ is a largest red component, $x\in R_i$, a contradiction. Then $N_R(x)\subseteq V_{3-i}\backslash R_{3-i}=V'_{3-i}$, and so $N_G(x)\cap R_{3-i}\subseteq N_B(x)$, that is $N_G(x)\cap(R'_{3-i}\uplus T_{3-i})\subseteq N_B(x)$.
\end{proof}

By Theorem \ref{EK},
\begin{equation}\label{b2}
|T_1|+|T_2|=|T|=\alpha'(\mathcal{R})<(1+\gamma)m.
\end{equation}
 Without loss of generality, assume that
\begin{equation}\label{b3}
|T_1|<\frac{1+\gamma}{2}m.
\end{equation}
For each $i\in[2]$, by the hypothesis, we have that
\begin{equation}\label{b5}
|R_i|\geq(1+\gamma)m,
\end{equation}
then $R'_i=R_i\backslash T_i\neq\emptyset$.
Now $|R'_1\uplus V'_1|=|V_1\backslash T_1|\overset{(\ref{b3})}{>}
\frac{1-\gamma}{2}m+n-1>(\frac{1}{2}-2\gamma)N$. By Claim \ref{b0}(i), $G[R'_1\uplus V'_1, R'_2]\subseteq G_B$. By Fact \ref{pr2}(i), $R'_2$ is contained in a blue component, say $\mathcal{B}$. For each $i\in[2]$, let $B_i=\mathcal{B}\cap V_i$. Then $R'_2\subseteq B_2$.

\begin{claim}\label{JT7}
The following holds.

{\rm (i)} $|B_1|\geq(1+\gamma)n$.

{\rm (ii)} $|T_2|>\frac{1-\gamma}{2}N$.

{\rm (iii)} $R'_2\uplus V'_2\subseteq B_2$.
\end{claim}
\begin{proof} (i) Recall that $R'_2\subseteq B_2$. Let $x\in R'_2$, by Claim \ref{b0}(i),
$|B_1|\geq|N_B(x)\cap(R'_1\uplus V'_1)|\geq\delta(G)-|T_1|
\overset{(\ref{b3})}{>}(\frac{1}{4}+\frac{\gamma}{2})m+(\frac{3}{4}+\gamma)(n-1)\geq (1+\gamma)n+\frac{1}{2}(\gamma n-\gamma-1)$ since $m\geq n+1$. If $\gamma=0$, then $|B_1|>n-\frac{1}{2}$, that is $|B_1|\geq n$. If $\gamma\neq0$, now $n\geq1+\frac{1}{\gamma}$, then $|B_1|>(1+\gamma)n$.

(ii) Suppose that $|T_2|\leq\frac{1-\gamma}{2}N$, then $|R'_2\uplus V'_2|=|V_2\backslash T_2|\geq\frac{1+\gamma}{2}N$. By Claim \ref{b0}(i), $G[R'_1, R'_2\uplus V'_2]\subseteq G_B$. By Fact \ref{pr2}(i), $R'_1$ is contained in a blue component of $G$.
For any $x\in R'_2\uplus V'_2$, by Claim \ref{b0},
$|N_B(x)\cap R'_1|\geq\delta(G)-|T_1\uplus V'_1|=\delta(G)-|T_1|-|V_1\setminus R_1|
>|R_1\setminus T_1|-(\frac{1}{4}-\gamma)N\underset{(\ref{b5})}{\overset{(\ref{b3})}{>}}
(\frac{1}{4}+\frac{3}{2}\gamma)m-(\frac{1}{4}-\gamma)(n-1)>0$ since $m\geq n+1$. Thus the blue component $\mathcal{B}$ contains both $R'_1$ and $R'_2\uplus V''_2$. Now $|B_2|\geq|R'_2\uplus V'_2|\geq\frac{1+\gamma}{2}N\geq(1+\gamma)n$ since $m\geq n+1$ and $|B_1|\geq(1+\gamma)n$ by (i), which is done.

(iii) By (ii), $|T_1|\overset{(\ref{b2})}{<}(1+\gamma)m-|T_2|<(1+\gamma)m-\frac{1-\gamma}{2}N$, then $|R'_1|=|R_1\backslash T_1|\overset{(\ref{b5})}{>}\frac{1-\gamma}{2}N$. By Claim \ref{b0}(i), $G[R'_1, R'_2\uplus V'_2]\subseteq G_B$. By Fact \ref{pr2}(i), the blue component $\mathcal{B}$ contains $R'_2\uplus V'_2$, and so $R'_2\uplus V'_2\subseteq B_2$.
\end{proof}

If $\gamma=0$, then $|B_1|\geq n$ by Claim \ref{JT7}(i) and $|B_2|\geq|R'_2\uplus V'_2|=|V_2\backslash T_2|\overset{(\ref{b2})}{\geq}n$ by Claim \ref{JT7}(iii), which is done.

Next we assume that $\gamma\neq0$, now $3n>m>n\geq1+\frac{1}{\gamma}$. Let $S$ be a minimum vertex cover of $\mathcal{B}$. For each $i\in[2]$, let $S_i=S\cap V_i$.

\begin{claim}\label{sta1} We have the following properties.

{\rm (i)} $|S_1|+|S_2|=|S|<(1+\gamma)n$.

{\rm (ii)} $R'_1\subseteq B_1\subseteq R_1$.

{\rm (iii)} $|R_1|>(\frac{3}{4}+\gamma)N$.

{\rm(iv)} $V_2=S_2\cup T_2$.
\end{claim}
\begin{proof} (i) If $\alpha'(\mathcal{B})\geq(1+\gamma)n$, implying that $\mathcal{B}$ has a matching of size $(1+\gamma)n$, then $|B_i|\geq(1+\gamma)n$ for each $i\in[2]$, which is done. Thus we can assume that $\alpha'(\mathcal{B})<(1+\gamma)n$. By Theorem \ref{EK},  $|S_1|+|S_2|=|S|=\alpha'(\mathcal{B})<(1+\gamma)n$.

(ii) For any $x\in R'_1$, by Claim \ref{b0}(i), $|N_B(x)\cap(R'_2\uplus V'_2)|\geq\delta(G)-|T_2|\overset{(\ref{b2})}{>}(\frac{3}{4}+\gamma)N-(1+\gamma)m\geq0$ since $3n>m>1+\frac{1}{\gamma}$. By Claim \ref{JT7}(iii), $R'_1\subseteq B_1$.

Suppose that $V'_1\cap B_1\neq\emptyset$. Let $x\in V'_1\cap B_1=(V_1\backslash R_1)\cap B_1$, then $N_R(x)\subseteq V_2\backslash R_2=V'_2\subseteq B_2$ by Claim \ref{JT7}(iii), and $N_B(x)\subseteq B_2$ since $\mathcal{B}$ is the blue component containing $B_1$. Thus $|B_2|\geq\delta(G)>(\frac{3}{4}+\gamma)N\geq(\frac{3}{2}+2\gamma)n$ since $m\geq n+1$ and $|B_1|\geq(1+\gamma)n$ by Claim \ref{JT7}(i), which is done. Thus we have that $V'_1\cap B_1=\emptyset$. Now $R'_1\subseteq B_1\subseteq V_1\backslash V'_1=R_1$.

(iii) Recall that $R'_2\neq\emptyset$. Let $x\in R'_2$. By Claim \ref{JT7}(iii), $N_B(x)\subseteq B_1\subseteq R_1$ by (ii). By Claim \ref{b0}(i), $N_R(x)\subseteq T_1$. Then $N_G(x)\subseteq R_1$, and so $|R_1|\geq\delta(G)>(\frac{3}{4}+\gamma)N$.

(iv) Suppose that $R'_2\backslash S_2\neq\emptyset$. Let $x\in R'_2\backslash S_2$. By Claim \ref{b0}(i), $N_R(x)\subseteq T_1$. By Claim \ref{JT7}(iii), $N_B(x)\subseteq S_1$ since $S$ is a minimum vertex cover of $\mathcal{B}$. Now $|S_1\cup T_1|\geq\delta(G)$, then $|S_1|\geq\delta(G)-|T_1|\overset{(\ref{b3})}{>}
(\frac{1}{4}+\frac{\gamma}{2})m+(\frac{3}{4}+\gamma)(n-1)\geq(1+\gamma)n$ since $m>n\geq1+\frac{1}{\gamma}$, a contradiction to (i). Thus $R'_2\subseteq S_2$.

Suppose that $V'_2\backslash S_2\neq\emptyset$. Let $x\in V'_2\backslash S_2$. Since $V'_2=V_2\backslash R_2$, $N_R(x)\subseteq V_1\backslash R_1=V'_1$. By Claim \ref{JT7}(iii), $N_B(x)\subseteq S_1$ since $S$ is a minimum vertex cover of $\mathcal{B}$. Now $|S_1\cup V'_1|\geq\delta(G)$. By (iii), $|V'_1|=|V_1\backslash R'_1|<(\frac{1}{4}-\gamma)N$. Then $|S_1|\geq\delta(G)-|V'_1|>(\frac{1}{2}+2\gamma)N\geq(1+4\gamma)n$ since $m\geq n+1$, a contradiction to (i). Thus $V'_2\subseteq S_2$.
Now $V_2=T_2\uplus R'_2\uplus V'_2=S_2\cup T_2$.
\end{proof}

Combining inequality (\ref{b2}) and Claim \ref{sta1}(i),
$|S\cup T|\leq|S|+|T|<(1+\gamma)(m+n)$.
By Claim \ref{sta1}(iv), $|S_1\cup T_1|=|S\cup T|-|S_2\cup T_2|=|S\cup T|-|V_2|<\gamma(m+n)+1$. By Claim \ref{sta1}(ii)-(iii), $|R'_1\backslash S_1|=|R_1\backslash(S_1\cup T_1)|
>(\frac{3}{4}+\gamma)N-|S_1\cup T_1|>\frac{3}{4}N-\gamma-1$.

Let $x\in R'_1\backslash S_1$. By Claim \ref{b0}(i), $N_R(x)\subseteq T_2$. Since $S$ is a minimum vertex cover of $\mathcal{B}$, $N_B(x)\subseteq S_2$. By Claim \ref{sta1}(iv), $|S_2\backslash T_2|=|V_2\backslash T_2|\overset{(\ref{b2})}{>}n-\gamma m-1$ and $|T_2\backslash S_2|=|V_2\backslash S_2|>m-\gamma n-1$ by Claim \ref{sta1}(i). Now the coloring is a $\gamma$-coloring as witnessed by $R'_1\backslash S_1\subseteq V_1$ and the partition $\{S_2\backslash T_2, S_2\cap T_2, T_2\backslash S_2\}$ of $V_2$.
\end{proof}

\begin{lemma}\label{4} Let $0\leq\gamma<\frac{1}{4}$ and let $G[V_1, V_2]$ be a balanced bipartite graph on $2(m+n-1)$ vertices with minimum degree $\delta(G)>(\frac{3}{4}+\gamma)(m+n-1)$, where $m>n$. Suppose that for some red-blue-edge-coloring of $G$ which is not $\gamma$-missing, there exists no blue connected matching of size $(1+\gamma)n$. Suppose that $\mathcal{B}$ is a largest blue component of G such that $|\mathcal{B}\cap V_i|\geq(1+\gamma)n$ for each $i\in[2]$. Let $S$ be a minimum vertex cover of $\mathcal{B}$, then $\mathcal{B}\setminus S$ is contained in a red component of $G$.
\end{lemma}
\begin{proof} For each $i\in[2]$, let $B_i=\mathcal{B}\cap V_i$, $S_i=S\cap V_i$, $B'_i=B_i\backslash S_i$ and $V'_i=V_i\backslash B_i$, then $V_i=S_i\uplus B'_i\uplus V'_i$. Let $N:=m+n-1$.

\begin{claim}\label{cc0} For each $i\in[2]$, the following holds.

{\rm (i)} If $x\in B'_i$, then $N_B(x)\subseteq S_{3-i}$ and $N_G(x)\cap(B'_{3-i}\uplus V'_{3-i})\subseteq N_R(x)$.

{\rm (ii)} If $x\in V'_i$, then $N_B(x)\subseteq V'_{3-i}$ and $N_G(x)\cap(B'_{3-i}\uplus S_{3-i})\subseteq N_R(x)$.

\end{claim}
\begin{proof} (i) Let $x\in B'_i=B_i\backslash S_i$. Since $S$ is a minimum vertex cover of $\mathcal{B}$, $N_B(x)\subseteq S_{3-i}$. Then $N_G(x)\cap(V_{3-i}\backslash S_{3-i})=N_G(x)\cap(B'_{3-i}\uplus V'_{3-i})\subseteq N_R(x)$.

(ii) Let $x\in V'_i=V_i\backslash B_i$. If $N_B(x)\cap B_{3-i}\neq\emptyset$, then since $\mathcal{B}$ is a largest blue component, $x\in B_i$, a contradiction. Then $N_B(x)\subseteq V_{3-i}\backslash B_{3-i}=V'_{3-i}$, and so $N_G(x)\cap B_{3-i}\subseteq N_R(x)$, that is $N_G(x)\cap(B'_{3-i}\uplus S_{3-i})\subseteq N_R(x)$.
\end{proof}

By Theorem \ref{EK},
\begin{equation}\label{c3}
|S_1|+|S_2|=|S|=\alpha'(B)<(1+\gamma)n.
\end{equation}
Without loss of generality, assume that
\begin{equation}\label{c4}
|S_1|<\frac{1+\gamma}{2}n.
\end{equation}
Let $i\in[2]$. By the hypothesis, we have that
\begin{equation}\label{cc1}
|B_i|\geq(1+\gamma)n,
\end{equation}
then $B'_i=B_i\backslash S_i\neq\emptyset$. %For any $x\in B'_i$, by Claim \ref{cc0}(i), we have that
%\begin{equation}\label{c5}
%|N_R(x)\cap(B'_{3-i}\uplus V'_{3-i})|\geq\delta(G)-|S_{3-i}|.
%\end{equation}
By Claim \ref{cc0}(i), $G[B'_i, B'_{3-i}\uplus V'_{3-i}]\subseteq G_R$. Note that
$|B'_{3-i}\uplus V'_{3-i}|=|V_{3-i}\backslash S_{3-i}|\overset{(\ref{c3})}{>}m-\gamma n-1\geq(\frac{1}{2}-2\gamma)N$ since $m\geq n+1$. By Fact \ref{pr2}(ii), $B'_i$ is contained in a red component of G, say $\mathcal{H}_i$.

If $\mathcal{H}_1=\mathcal{H}_2$, then we are done. Thus we assume that $\mathcal{H}_1\neq\mathcal{H}_2$, that is $\mathcal{H}_1\cap \mathcal{H}_2=\emptyset$ and $G_R[V(\mathcal{H}_1), V(\mathcal{H}_2)]=\emptyset$. For each $i\in[2]$, let $S^1_i=\mathcal{H}_1\cap S_i$, $S^2_i=\mathcal{H}_2\cap S_i$ and $S^3_i=S_i\backslash(S^1_i\uplus S^2_i)$; and
let $C^1_i=\mathcal{H}_1\cap V'_i$, $C^2_i=\mathcal{H}_2\cap V'_i$ and
$C^3_i=V'_i\backslash(C^1_1\uplus C^2_2)$.

Let $x\in B'_1$, $N_R(x)\subseteq\mathcal{H}_1\cap V_2=S^1_2\uplus C^1_2$ and by Claim \ref{cc0}(i), $N_B(x)\subseteq S_2$, thus $N_G(x)\subseteq S_2\uplus C^1_2$. That is
\begin{equation}\label{c7}
|S_2|+|C^1_2|\geq\delta(G)>(\frac{3}{4}+\gamma)N.
\end{equation}
Let $x\in B'_2$, $N_R(x)\subseteq \mathcal{H}_2\cap V_1=S^2_1\uplus C^2_1$ and by Claim \ref{cc0}(i), $N_B(x)\subseteq S_1$, thus $N_G(x)\subseteq S_1\uplus C^2_1$. That is
\begin{equation}\label{c9}
|S_1|+|C^2_1|\geq\delta(G)>(\frac{3}{4}+\gamma)N.
\end{equation}
Combining with $m\geq n+1$, we have that
\begin{equation}\label{c10}
|C^2_1|>(\frac{3}{4}+\gamma)N-|S_1|
\overset{(\ref{c4})}{>}(\frac{3}{4}+\gamma)(m-1)+(\frac{1}{4}+\frac{\gamma}{2})n\geq\frac{N}{2}.
\end{equation}

\begin{claim}\label{JT6}
For each $i\in[2]$, $\mathcal{H}_i\cap V_i=B_i\uplus C^i_i$ and $\mathcal{H}_i\cap V_{3-i}=C^i_{3-i}$.
\end{claim}
\begin{proof} Let $x\in C^1_2=\mathcal{H}_1\cap V'_2$. By Claim \ref{cc0}(ii), $N_B(x)\subseteq V'_1$. Since $\mathcal{H}_1$ is the red component containing $C^1_2$, $N_R(x)\subseteq \mathcal{H}_1\cap V_1=B'_1\uplus S^1_1\uplus C^1_1$. Then $N_G(x)\cap(S^2_1\uplus S^3_1)=\emptyset$. Thus $G[S^2_1\uplus S^3_1, C^1_2]=\emptyset$. Note that
$|C^1_2|\overset{(\ref{c7})}{>}(\frac{3}{4}+\gamma)N-|S_2|
\overset{(\ref{c3})}{>}(\frac{3}{4}+\gamma)(m-1)-\frac{n}{4}\geq(\frac{1}{4}-\gamma)N$ since $m\geq n+1$. Then $S^2_1\uplus S^3_1=\emptyset$, otherwise by Fact \ref{pr1}, $|C^1_2|<(\frac{1}{4}-\gamma)N$, a contradiction. Now $S_1\subseteq \mathcal{H}_1\cap V_1$, and so $B_1\subseteq \mathcal{H}_1\cap V_1$.
Then $\mathcal{H}_1\cap V_1=B_1\uplus C^1_1$. Since $\mathcal{H}_1\cap \mathcal{H}_2=\emptyset$, $\mathcal{H}_2\cap V_1=C^2_1$.

Let $x\in C^2_1=\mathcal{H}_2\cap V'_1$. In a similar way to the above, we can get that
$\mathcal{H}_2\cap V_2=B_2\uplus C^2_2$ and $\mathcal{H}_1\cap V_2=C^1_2$.
\end{proof}

We split our argument into two cases.

{\bf Case 1.} $|C^1_2|\leq\frac{N}{2}$.

Note that $|S_2|\overset{(\ref{c7})}{>}(\frac{3}{4}+\gamma)N-|C^1_2|\geq(\frac{1}{4}+\gamma)N$, then
$|S_1|\overset{(\ref{c3})}{<}(1+\gamma)n-|S_2|<\frac{3}{4}n-(\frac{1}{4}+\gamma)(m-1)$,
implying that $m-1<\frac{3n}{1+4\gamma}$. Now
$|C^2_1|\overset{(\ref{c9})}{>}(\frac{3}{4}+\gamma)N-|S_1|>(1+2\gamma)(m-1)+\gamma n.$

By Claim \ref{JT6}, $B_2\subseteq \mathcal{H}_2\cap V_2$. For any $x\in B_2$, $N_B(x)\subseteq B_1$ and by Claim \ref{JT6}, $N_R(x)\subseteq \mathcal{H}_2\cap V_1=C^2_1$, thus $N_G(x)\cap(V'_1\backslash C_1^2)=N_G(x)\cap(C^1_1\uplus C^3_1)=\emptyset$. That is $G[C^1_1\uplus C^3_1, B_2]=\emptyset$. Now $C^1_1\uplus C^3_1=\emptyset$, otherwise by Fact \ref{pr1}, $|B_2|<(\frac{1}{4}-\gamma)N<(1+\gamma)n$ since $m-1<\frac{3n}{1+4\gamma}$, a contradiction to inequality (\ref{cc1}). Thus $V'_1=C^2_1$, now $|V_1|=|B_1\uplus C^2_1|\overset{(\ref{cc1})}{\geq}(1+\gamma)n+|C^2_1|>(1+\gamma)n+(1+2\gamma)(m-1)+\gamma n=(1+2\gamma)N$, a contradiction.

{\bf Case 2.} $|C^1_2|>\frac{N}{2}$.

Let $x\in C^2_1=\mathcal{H}_2\cap V'_1$. By Claim \ref{cc0}(ii), $N_B(x)\subseteq V'_2=C^1_2\uplus C^2_2\uplus C^3_2$. By Claim \ref{JT6}, $N_R(x)\subseteq \mathcal{H}_2\cap V_2=B_2\uplus C^2_2$. Thus $G[C^2_1, C^1_2\uplus C^3_2]\subseteq G_B$. Since $|C^1_2|>\frac{N}{2}$, by Fact \ref{pr2}(i), $C^2_1$ is contained in a blue component of $G$, say $\mathcal{H}$.

Let $x\in C^1_2=\mathcal{H}_1\cap V'_2$. By Claim \ref{cc0}(ii), $N_B(x)\subseteq V'_1=C^1_1\uplus C^2_1\uplus C^3_1$. By Claim \ref{JT6}, $N_R(x)\subseteq V(\mathcal{H}_1)\cap V_1=B_1\uplus C^1_1$. Then
$|N_B(x)\cap C^2_1|\geq\delta(G)-|V_1\backslash C^2_1|>|C^2_1|-(\frac{1}{4}-\gamma)N\overset{(\ref{c10})}{>}(\frac{1}{4}+\gamma)N$.
Since $C^2_1\subseteq \mathcal{H}\cap V_1$, $C^1_2\subseteq \mathcal{H}\cap V_2$. Now $|\mathcal{H}\cap V_1|\geq|C^2_1|\overset{(\ref{c10})}{>}\frac{N}{2}\geq|V_1\backslash C^2_1|\geq|B_1|$. By the assumption $|C^1_2|>\frac{N}{2}$, $|\mathcal{H}\cap V_2|\geq|C^1_2|>\frac{N}{2}\geq|V_2\backslash C^1_2|\geq|B_2|$, a contradiction to the maximality of $\mathcal{B}$.
\end{proof}

\noindent{\bf Proof of Theorem \ref{1}.} Suppose that there exists a red-blue-edge-coloring of $G$ yielding neither red connected $m$-matching nor blue connected $n$-matching. Let $\gamma=0$ and $n_0=1$. By Lemma \ref{3}, we can assume that $\mathcal{B}$ is a largest blue component such that $|\mathcal{B}\cap V_i|\geq n$ for each $i\in[2]$. Let $S$ be a minimum vertex cover of $\mathcal{B}$. By Lemma \ref{4}, $\mathcal{B}\backslash S$ is contained in some red component of $G$, say $\mathcal{R}$. Let $T$ be a minimum vertex cover of $\mathcal{R}$. For each $i\in[2]$, let $B_i=\mathcal{B}\cap V_i$, $S_i=S\cap V_i$, $B'_i=B_i\backslash S_i$ and $V'_i=V_i\backslash B_i$. For each $i\in[2]$, let $R_i=\mathcal{R}\cap V_i$, $T_i=T\cap V_i$, $T'_i=T_i\backslash B_i$, $R'_i=R_i\backslash(B_i\cup T_i)$, and $V''_i=V_i\backslash(B_i\cup R_i)$. For each $i\in[2]$, $B'_i\subseteq B_i\cap R_i$ and $V'_i=R'_i\uplus T'_i\uplus V''_i$.

\begin{claim}\label{d0} For each $i\in[2]$, the following holds.

{\rm (i)} If $x\in B'_i\backslash T_i$, then $N_B(x)\subseteq S_{3-i}$ and $N_R(x)\subseteq T_{3-i}$.

{\rm (ii)} If $x\in R'_i$, then $N_B(x)\subseteq T'_{3-i}\uplus R'_{3-i}\uplus V''_{3-i}$ and $N_R(x)\subseteq T_{3-i}$.

{\rm (iii)} If $x\in V''_i$, then $N_B(x)\subseteq T'_{3-i}\uplus R'_{3-i}\uplus V''_{3-i}$ and $N_R(x)\subseteq(S_{3-i}\backslash R_{3-i})\uplus V''_{3-i}$.

{\rm (iv)} If $x\in S_i\backslash R_i$, then $N_B(x)\subseteq B_{3-i}$ and $N_R(x)\subseteq(S_{3-i}\backslash R_{3-i})\uplus V''_{3-i}$.

{\rm (v)} If $x\in (S_i\cap R_i)\backslash T_i$, then $N_B(x)\subseteq B_{3-i}$ and $N_R(x)\subseteq T_{3-i}$.

{\rm (vi)} $G[V''_i, B_{3-i}\cap R_{3-i}]=\emptyset$.

{\rm (vii)} $G[S_i\backslash R_i, T'_{3-i}\uplus R'_{3-i}]=\emptyset$.
\end{claim}
\begin{proof} (i) Let $x\in B'_i\backslash T_i\subseteq(B_i\cap R_i)\backslash T_i$. Since $T$ is a minimum vertex cover of $\mathcal{R}$, $N_R(x)\subseteq T_{3-i}$. Since $S$ is a minimum vertex cover of $\mathcal{B}$, $N_B(x)\subseteq S_{3-i}$.

(ii) Let $x\in R'_i=R_i\backslash(B_i\cup T_i)$. If $N_B(x)\cap B_{3-i}\neq\emptyset$, then since $\mathcal{B}$ is a largest blue component, $x\in B_i$, a contradiction. Thus $N_B(x)\subseteq V_{3-i}\backslash B_{3-i}=T'_{3-i}\uplus R'_{3-i}\uplus V''_{3-i}$. Since $T$ is a minimum vertex cover of $\mathcal{R}$, $N_R(x)\subseteq T_{3-i}$.

(iii) Let $x\in V''_i=V_i\backslash(B_i\cup R_i)$. If $N_B(x)\cap B_{3-i}\neq\emptyset$, then since $\mathcal{B}$ is a largest blue component, $x\in B_i$, a contradiction. Thus $N_B(x)\subseteq V_{3-i}\backslash B_{3-i}=T'_{3-i}\uplus R'_{3-i}\uplus V''_{3-i}$. If $N_R(x)\cap R_{3-i}\neq\emptyset$, then since $\mathcal{R}$ is a largest red component, $x\in R_i$, a contradiction. Thus $N_R(x)\subseteq V_{3-i}\backslash R_{3-i}=(S_{3-i}\backslash R_{3-i})\uplus V''_{3-i}$.

(iv) Let $x\in S_i\backslash R_i\subseteq B_i\backslash R_i$. Since $\mathcal{B}$ is a largest blue component, $N_B(x)\subseteq B_{3-i}$. If $N_R(x)\cap R_{3-i}\neq\emptyset$, then since $\mathcal{R}$ is a largest red component, $x\in R_i$, a contradiction. Thus $N_R(x)\subseteq V_{3-i}\backslash R_{3-i}=(S_{3-i}\backslash R_{3-i})\uplus V''_{3-i}$.

(v) Let $x\in (S_i\cap R_i)\backslash T_i$. Since $\mathcal{B}$ is a largest blue component, $N_B(x)\subseteq B_{3-i}$. Since $T$ is a minimum vertex cover of $\mathcal{R}$, $N_R(x)\subseteq T_{3-i}$.

(vi) If $V''_i=\emptyset$ or $B_{3-i}\cap R_{3-i}=\emptyset$, then we are done. Suppose that $V''_i\neq\emptyset$ and $B_{3-i}\cap R_{3-i}\neq\emptyset$. For any $x\in V''_{3-i}$, by (iii), $N_B(x)\subseteq V_{3-i}\backslash B_{3-i}$ and $N_R(x)\subseteq V_{3-i}\backslash R_{3-i}$, then $N_G(x)\cap(B_{3-i}\cap R_{3-i})=\emptyset$. Thus $G[V''_i, B_{3-i}\cap R_{3-i}]=\emptyset$.

(vii) If $S_i\backslash R_i=\emptyset$ or $T'_{3-i}\uplus R'_{3-i}=\emptyset$, then we are done. Suppose that $S_i\backslash R_i\neq\emptyset$ and $T'_{3-i}\uplus R'_{3-i}\neq\emptyset$. For any $x\in S_i\backslash R_i$, by (iv), $N_B(x)\subseteq B_{3-i}$ and $N_R(x)\subseteq(S_{3-i}\backslash R_{3-i})\uplus V''_{3-i}$, then $N_G(x)\cap(T'_{3-i}\uplus R'_{3-i})=\emptyset$. Thus $G[S_i\backslash R_i, T'_{3-i}\uplus R'_{3-i}]=\emptyset$.
\end{proof}

By Theorem \ref{EK},
\begin{equation}\label{d4}
|S_1|+|S_2|=|S|=\alpha'(\mathcal{B})\leq n-1,
\end{equation}
and
\begin{equation}\label{d3}
|T_1|+|T_2|=|T|=\alpha'(\mathcal{R})\leq m-1.
\end{equation}
Then
\begin{equation}\label{d5}
|S\cup T|\leq|S|+|T|\leq m+n-2.
\end{equation}
Let $i\in[2]$. By the hypothesis, we have that
\begin{equation}\label{d1}
|B_i|=|S_i|+|B'_i|\geq n,
\end{equation}
then $B'_i=B_i\backslash S_i\neq\emptyset$. Since $B'_i\subseteq B_i\cap R_i$, by Claim \ref{d0}(vi) and Proposition \ref{pr1},
\begin{equation}\label{d8}
|V''_i|<\frac{m+n-1}{4}.
\end{equation}

Let $i\in[2]$. If $x\in B'_i\backslash T_i$, then by Claim \ref{d0}(i), $N_G(x)\subseteq S_{3-i}\cup T_{3-i}$, so $|S_{3-i}\cup T_{3-i}|\geq\delta(G)$. If $B'_i\backslash T_i\neq\emptyset$ for each $i\in[2]$, then $|S\cup T|=|S_1\cup T_1|+|S_2\cup T_2|\geq2\delta(G)>\frac{3}{2}(m+n-1)$, a contradiction to inequality (\ref{d5}). Thus either $B'_1\subseteq T_1$ or $B'_2\subseteq T_2$. Without loss of generality, assume that
\begin{equation}\label{JT1}
B'_1\subseteq T_1.
\end{equation}
Then $B_1\cup T_1=S_1\cup T_1$, and so
\begin{equation}\label{(ii)}
|R'_1\uplus V''_1|=|V_1\backslash(B_1\cup T_1)|=|V_1\backslash(S_1\cup T_1)|\overset{(\ref{d5})}{\geq}1.
\end{equation}
By (\ref{JT1}), we have that
\begin{equation}\label{91}
V_1=(S_1\backslash T_1)\uplus T_1\uplus R'_1\uplus V''_1.
\end{equation}
Let $B''_2=B'_2\backslash T_2$, $ST_2=S_2\cap T_2$, $SR_2=(S_2\cap R_2)\backslash T_2$ and $S'_2=S_2\backslash R_2$. Then
\begin{equation}\label{92}
V_2=S'_2\uplus SR_2\uplus B''_2\uplus T_2\uplus R'_2\uplus V''_2.
\end{equation}

For any $x\in R'_1$, $N_G(x)\subseteq T_2\uplus R'_2\uplus V''_2$ by Claim \ref{d0}(ii), then $N_G(x)\cap(S'_2\uplus SR_2\uplus B''_2)=\emptyset$ by (\ref{92}). Thus $G[R'_1, S'_2\uplus SR_2\uplus B''_2]=\emptyset$. Since $B_2\cap R_2=ST_2\uplus SR_2\uplus B'_2$, $G[V''_1, ST_2\uplus SR_2\uplus B'_2]=\emptyset$ by Claim \ref{d0}(vi). Thus $G[R'_1\uplus V''_1, SR_2\uplus B''_2]=\emptyset$. By (\ref{(ii)}) and Proposition \ref{pr1}, we have that
\begin{equation}\label{105}
|B''_2\uplus SR_2|<\frac{m+n-1}{4}.
\end{equation}

Next we split our argument into two cases.

{\bf Case 1.} $B'_1\subseteq T_1$ and $B'_2\backslash T_2\neq\emptyset$.

Now $B''_2=B'_2\backslash T_2\neq\emptyset$. For any $x\in B''_2$, $N_G(x)\subseteq S_1\cup T_1$ by Claim \ref{d0}(i). Then
\begin{equation}\label{d9}
|S_1\cup T_1|\geq\delta(G)>\frac{3}{4}(m+n-1),
\end{equation}
and so
\begin{equation}\label{d11}
|S_2\cup T_2|=|S\cup T|-|S_1\cup T_1|\overset{(\ref{d5})}{<}\frac{m+n-5}{4}.
\end{equation}
Now $|B_2\cup T_2|=|S'_2\uplus SR_2\uplus B''_2\uplus T_2|=|S'_2\uplus T_2|+|SR_2\uplus B''_2|
\overset{(\ref{105})}{<}|S_2\cup T_2|+\frac{m+n-1}{4}\overset{(\ref{d11})}{<}\frac{m+n-3}{2}$, then
\begin{equation}\label{d12}
|R'_2|=|V_2|-|B_2\cup T_2|-|V''_2|\overset{(\ref{d8})}{>}\frac{m+n+3}{4}.
\end{equation}
By (\ref{JT1}), $B_1\cup T_1=S_1\cup T_1$, then
\begin{equation}\label{d10}
|R'_1\uplus V''_1|=|V_1\backslash(B_1\cup T_1)|=|V_1\backslash(S_1\cup T_1)|\overset{(\ref{d9})}{<}\frac{m+n-1}{4}.
\end{equation}

For any $x\in R'_2$, $N_B(x)\subseteq T'_1\uplus R'_1\uplus V''_1$ and $N_R(x)\subseteq T_1$ by Claim \ref{d0}(ii), then $N_G(x)\cap(S_1\backslash T_1)=\emptyset$ by (\ref{91}). So $G[S_1\backslash T_1, R'_2]=\emptyset$. If $S_1\backslash T_1\neq\emptyset$, then by Proposition \ref{pr1}, $|R'_2|<\frac{m+n-1}{4}$, contradicting to inequality (\ref{d12}). Then $S_1\subseteq T_1$, and so $B_1\subseteq T_1$ by (\ref{JT1}). Now $V_1=T_1\uplus R'_1\uplus V''_1$, thus
\begin{equation}\label{d13}
|R'_1\uplus V''_1|=|V_1|-|T_1|\overset{(\ref{d3})}{\geq}n.
\end{equation}
Combining with inequality (\ref{d10}), we have that
\begin{equation}\label{d14}
m\geq 3n+2.
\end{equation}%so $|T_1|>\frac{3}{4}(m+n-1)$, then $|T_2|<\frac{m-3n-1}{4}$,
By the assumption $B''_2\neq\emptyset$, we have that
\begin{equation}\label{d17}
|S'_2\uplus T_2\uplus V''_2|\overset{(\ref{92})}{=}
|V_2\backslash(SR_2\uplus B''_2\uplus R'_2)|<|V_2\backslash R'_2|.
\end{equation}

\begin{claim}\label{JT4}
$R'_1\uplus V''_1$ is contained in some blue component of $G$, say $\mathcal{H}_1$.
\end{claim}
\begin{proof} For any $x\in R'_1$, $N_B(x)\subseteq T'_2\uplus R'_2\uplus V''_2$ and $N_R(x)\subseteq T_2$ by Claim \ref{d0}(ii), then
\[|N_B(x)\cap R'_2|\geq\delta(G)-|T_2\uplus V''_2|.\]
For any $x\in V''_1$, $N_B(x)\subseteq T'_2\uplus R'_2\uplus V''_2$ and $N_R(x)\subseteq S'_2\uplus V''_2$ by Claim \ref{d0}(iii), then
\[|N_B(x)\cap R'_2|\geq\delta(G)-|S'_2\uplus T'_2\uplus V''_2|.\]
For any pair of vertices $x, x'\in R'_1\uplus V''_1$, by the inclusion-exclusion principle, we have that
\begin{align}\label{201}
|N_B(x)\cap N_B(x')|&\geq|N_B(x)\cap R'_2|+|N_B(x')\cap R'_2|-|R'_2|
\geq2(\delta(G)-|S'_2\uplus T_2\uplus V''_2|)-|R'_2| \notag\\
&\geq2(\delta(G)-|S_2\cup T_2|-|V''_2|)-|R'_2|
\overset{(\ref{d8})}{\underset{(\ref{d11})}{>}}\frac{m+n+3}{2}-|R'_2|,
\end{align}
and
\begin{align}\label{202}
|N_B(x)\cap N_B(x')|&\geq|N_B(x)\cap R'_2|+|N_B(x')\cap R'_2|-|R'_2|
\geq2(\delta(G)-|S'_2\uplus T_2\uplus V''_2|)-|R'_2| \notag\\
&\overset{(\ref{d17})}{>}2(\delta(G)-|V_2\backslash R'_2|)-|R'_2| \notag\\
&=2\delta(G)-2|V_2|+|R'_2|>|R'_2|-\frac{m+n-1}{2}.
\end{align}

If $|R'_2|<\frac{m+n+3}{2}$, then the rightside of inequality (\ref{201}) is at least 1. If $|R'_2|\geq\frac{m+n+3}{2}$, then the rightside of inequality (\ref{202}) is at least 1. Thus we have that $R'_1\uplus V''_1$ is contained in some blue component of $G$, say $\mathcal{H}_1$.
\end{proof}

Let $J$ be a minimum vertex cover of $\mathcal{H}_1$. For each $i\in[2]$, let $J_i=J\cap V_i$. By Theorem \ref{EK},
\begin{equation}\label{d18}
|J_1|+|J_2|=|J|=\alpha'(\mathcal{H}_1)\leq n-1.
\end{equation}
By Claim \ref{JT4}, $R'_1\uplus V''_1\subseteq\mathcal{H}_1\cap V_1$. Now $(R'_1\uplus V''_1)\backslash J_1\neq\emptyset$, otherwise $|J_1|\geq|R'_1\uplus V''_1|\overset{(\ref{d13})}{\geq}n$, a contradiction to inequality (\ref{d18}).
Let $x\in(R'_1\uplus V''_1)\backslash J_1$. Since $J$ is a minimum vertex cover of $\mathcal{H}_1$, $N_B(x)\subseteq J_2$. By Claim \ref{d0}(ii)-(iii), $N_R(x)\subseteq S'_2\uplus T_2\uplus V''_2$. Then $|J_2\cup(T_2\uplus S'_2\uplus V''_2)|\geq\delta(G)$, and so
$|J_2|\geq\delta(G)-|T_2\uplus S'_2\uplus V''_2|
\geq\delta(G)-|T_2\cup S_2|-|V''_2|\overset{(\ref{d8})}{\underset{(\ref{d11})}{>}}
\frac{m+n+3}{4}\overset{(\ref{d14})}{\geq}n+\frac{5}{4}$, a contradiction to inequality (\ref{d18}).

{\bf Case 2.} For each $i\in[2]$, $B'_i\subseteq T_i$.

For each $i\in[2]$, $B_i\subseteq S_i\cup T_i$.
If $x\in S_i\backslash R_i$, by Claim \ref{d0}(iv), $N_G(x)\subseteq B_{3-i}\uplus V''_{3-i}$, then
$|B_{3-i}\uplus V''_{3-i}|\geq\delta(G)$, thus
\[|S_{3-i}\cup T_{3-i}|\geq|B_{3-i}|\geq\delta(G)-|V''_{3-i}|
\overset{(\ref{d8})}{>}\frac{m+n-1}{2}.\]
If $x\in(S_i\cap R_i)\backslash T_i$, by Claim \ref{d0}(v), $N_G(x)\subseteq B_{3-i}\cup T_{3-i}$, then
\[|S_{3-i}\cup T_{3-i}|=|B_{3-i}\cup T_{3-i}|\geq\delta(G)>\frac{3}{4}(m+n-1).\]
If $S_i\backslash T_i=(S_i\backslash R_i)\uplus((S_i\cap R_i)\backslash T_i)\neq\emptyset$ for each $i\in[2]$, then each $|S_i\cup T_i|>\frac{m+n-1}{2}$, and so $|S\cup T|=|S_1\cup T_1|+|S_2\cup T_2|\geq m+n$, a contradiction to inequality (\ref{d5}).
Thus either $S_1\subseteq T_1$ or $S_2\subseteq T_2$. Without loss of generality, assume that $S_1\subseteq T_1$. Combining with (\ref{JT1}), we have that
\begin{equation}\label{JT2}
B_1\subseteq T_1.
\end{equation}
Then $V_1=T_1\uplus R'_1\uplus V''_1$, and so
\begin{equation}\label{d22}
|R'_1\uplus V''_1|=|V_1|-|T_1|\overset{(\ref{d3})}{\geq}n.
\end{equation}

Now we split the remainder into two cases.

{\bf Subcase~2.1.} $S_1\subset T_1$ and $S_2\backslash T_2\neq\emptyset$.

Recall that $S'_2=S_2\backslash R_2$ and $SR_2=(S_2\cap R_2)\backslash T_2$, then $S'_2\uplus SR_2\neq\emptyset$.

\begin{claim}\label{106}
$S'_2\neq\emptyset$.
\end{claim}

\begin{proof} Suppose that $S'_2=\emptyset$, then $SR_2\neq\emptyset$. For any $x\in SR_2$, by Claim \ref{d0}(v) and (\ref{JT2}), $N_G(x)\subseteq T_1$, then $N_G(x)\cap(V_1\backslash T_1)=N_G(x)\cap(R'_1\uplus V''_1)=\emptyset$. Thus $G[R'_1\uplus V''_1, SR_2]=\emptyset$. By Proposition \ref{pr1}, $|R'_1\uplus V''_1|<\frac{m+n-1}{4}$,
then $|T_1|=|V_1\backslash(R'_1\uplus V''_1)|>\frac{3}{4}(m+n-1)$. Now
\begin{equation}\label{d25}
|T_2|\overset{(\ref{d3})}{\leq}m-1-|T_1|<\frac{m-3n-1}{4}.
\end{equation}
Since $B'_2\subseteq T_2$ and $S'_2=\emptyset$,
$|R'_2|\overset{(\ref{92})}{=}|V_2\backslash V''_2|-|S_2\cup T_2|\overset{(\ref{d8})}{>}
\frac{3}{4}(m+n-1)-|S_2\cup T_2|\overset{(\ref{d4})}{\underset{(\ref{d25})}{>}}\frac{m+n-1}{2}$. By Claim \ref{d0}(ii), $G[R'_1\uplus V''_1, R'_2]\subseteq G_B$.
By Proposition \ref{pr2}(i), $R'_1\uplus V''_1$ is contained in some blue component of $G$, say $\mathcal{H}_2$. Then $R'_1\uplus V''_1\subseteq\mathcal{H}_2\cap V_1$.

Let $K$ be a minimum vertex cover of $\mathcal{H}_2$. For each $i\in[2]$, let $K_i=K\cap V_i$. By Theorem \ref{EK},
\begin{equation}\label{d28}
|K_1|+|K_2|=|K|=\alpha'(\mathcal{H}_2)\leq n-1.
\end{equation}
Now $(R'_1\uplus V''_1)\backslash K_1\neq\emptyset$, otherwise $|K_1|\geq|R'_1\uplus V''_1|\overset{(\ref{d22})}{\geq}n$, a contradiction to inequality (\ref{d28}). Let $x\in(R'_1\uplus V''_1)\backslash K_1$. By Claim \ref{d0}(ii)-(iii), $N_R(x)\subseteq T_2\uplus V''_2$ since $S'_2=\emptyset$. Since $K$ is a minimum vertex cover of $\mathcal{H}_2$, $N_B(x)\subseteq K_2$. Then $|K_2\cup(T_2\uplus V''_2)|\geq\delta(G)$, and so $|K_2|\geq\delta(G)-|T_2\uplus V''_2|
\overset{(\ref{d8})}{\underset{(\ref{d25})}{>}}\frac{m+5n-1}{4}\geq\frac{3}{2}n$ since $m\geq n+1$,
a contradiction to inequality (\ref{d28}).

This completes the proof of Claim \ref{106}.
\end{proof}

By Claim \ref{106}, $S'_2\neq\emptyset$. By Claim \ref{d0}(vii) and Proposition \ref{pr1}, $|T'_1\uplus R'_1|<\frac{m+n-1}{4}$. By (\ref{JT2}), $|B_1\uplus V''_1|=|V_1\backslash(T'_1\uplus R'_1)|>\frac{3}{4}(m+n-1)$, then
\begin{equation}\label{d29}
|T_1|\overset{(\ref{JT2})}{\geq}|B_1|>\frac{3}{4}(m+n-1)-|V''_1|\overset{(\ref{d8})}{>}\frac{m+n-1}{2}.
\end{equation}
Then
\begin{equation}\label{d30}
|T_2|\overset{(\ref{d3})}{\leq}m-1-|T_1|<\frac{m-n-1}{2}.
\end{equation}

\begin{claim}\label{JT8}
The following holds.

{\rm(i)} $V''_2=\emptyset$.

{\rm (ii)} $R'_1\uplus V''_1$ is contained in some blue component of $G$, say $\mathcal{H}_3$.
\end{claim}
\begin{proof}(i) For any $x\in V''_2$, by Claim \ref{d0}(iii) and (\ref{JT2}), $N_G(x)\subseteq T'_1\uplus R'_1\uplus V''_1=V_1\backslash B_1$. Thus $G[B_1, V''_2]=\emptyset$. If $V''_2\neq\emptyset$, by Proposition \ref{pr1}, $|B_1|<\frac{m+n-1}{4}$, a contradiction to inequality (\ref{d29}). Thus $V''_2=\emptyset$.

(ii) By (i), $V_2\overset{(\ref{92})}{=}S'_2\uplus SR_2\uplus T_2\uplus R'_2$ since $B'_2\subseteq T_2$. Then $|R'_2|=|V_2\backslash(S_2\cup T_2)|\overset{(\ref{d4})}{\geq}m-|T_2|
\overset{(\ref{d30})}{>}\frac{m+n+1}{2}$.
By Claim \ref{d0}(ii), $G[R'_1\uplus V''_1, R'_2]\subseteq G_B$. By Proposition \ref{pr2}(i), $R'_1\uplus V''_1$ is contained in some blue component of $G$, say $\mathcal{H}_3$.
\end{proof}

Let $D$ be a minimum vertex cover of $\mathcal{H}_3$. For each $i\in[2]$, let $D_i=D\cap V_i$. By Theorem \ref{EK},
\begin{equation}\label{d32}
|D_1|+|D_2|=|D|=\alpha'(\mathcal{H}_3)\leq n-1.
\end{equation}
By Claim \ref{JT8}(ii), $R'_1\uplus V''_1\subseteq\mathcal{H}_3\cap V_1$. Now $(R'_1\uplus V''_1)\backslash D_1\neq\emptyset$, otherwise $|D_1|\geq|R'_1\uplus V''_1|\overset{(\ref{d22})}{\geq}n$, a contradiction to inequality (\ref{d32}). Let $x\in(R'_1\uplus V''_1)\backslash D_1$. Since $D$ is a minimum vertex cover of $\mathcal{H}_3$, $N_B(x)\subseteq D_2$. By Claim \ref{d0}(ii)-(iii) and Claim \ref{JT8}(i), $N_R(x)\subseteq S'_2\uplus T_2$. Then $|D_2\cup(T_2\uplus S'_2)|\geq\delta(G)$, and so
$|D_2|\geq\delta(G)-|T_2\uplus S'_2|\overset{(\ref{d30})}{>}\frac{m+5n-1}{4}-|S'_2|$.
Combining with inequality (\ref{d32}), then
\[|S'_2|>\frac{m+n+3}{4}.\]
Combining with inequality (\ref{d4}), we have that $m\leq 3n-8$.
Now $T'_1\uplus R'_1=\emptyset$, otherwise by Claim \ref{d0}(vii) and Proposition \ref{pr1}, $|S'_2|<\frac{m+n-1}{4}$, a contradiction. Thus $|V''_1|\overset{(\ref{d22})}{\geq}n$. Combining with inequality (\ref{d8}), we have that $m\geq 3n+2$, a contradiction.

{\bf Subcase~2.2.} For each $i\in[2]$, $S_i\subset T_i$.

For each $i\in[2]$, we have that $B_i\subseteq T_i$, then
\begin{equation}\label{d35}
V_i=T_i\uplus R'_i\uplus V''_i=B_i\uplus T'_i\uplus R'_i\uplus V''_i.
\end{equation}
%and so
%\begin{equation}\label{d36}
%|R_i|=|T_i\uplus R'_i|=|V_i\backslash V''_i|\overset{(\ref{d8})}{>}\frac{3}{4}(m+n-1).
%\end{equation}
Without loss of generality, by inequality (\ref{d3}), we can assume that
\begin{equation}\label{d37}
|T_1|\leq\frac{m-1}{2}.
\end{equation}

\begin{claim}\label{11}
The following holds.

{\rm (i)} $R'_2$ is contained in some blue component of $G$, say $\mathcal{H}$.

{\rm (ii)} $|R'_2|\leq n-1$.

{\rm (iii)} For each $i\in[2]$, $V''_i\neq\emptyset$.
\end{claim}
\begin{proof}(i) By Claim \ref{d0}(ii), $G[R'_1\uplus V''_1, R'_2]\subseteq G_B$.
Note that $|R'_1\uplus V''_1|\overset{(\ref{d35})}{=}|V_1\backslash T_1|
\overset{(\ref{d37})}{\geq}\frac{m-1}{2}+n$. By Proposition \ref{pr2}(i), $R'_2$ is contained in some blue component of $G$, say $\mathcal{H}$.

Let $U$ be a minimum vertex cover of $\mathcal{H}$. For each $i\in[2]$, let $U_i=U\cap V_i$. By Theorem \ref{EK},
\begin{equation}\label{d41}
|U_1|+|U_2|=|U|=\alpha'(\mathcal{H})\leq n-1.
\end{equation}

(ii) By (i), $R'_2\subseteq\mathcal{H}\cap V_2$. Suppose that $|R'_2|\geq n$, then $R'_2\backslash U_2\neq\emptyset$, otherwise $|U_2|\geq|R'_2|\geq n$, a contradiction to inequality (\ref{d41}). Let $x\in R'_2\backslash U_2$. Since $U$ is a minimum vertex cover of $\mathcal{H}$, $N_B(x)\subseteq U_1$. By Claim \ref{d0}(ii), $N_R(x)\subseteq T_1$. Then $|U_1\cup T_1|\geq\delta(G)$, and so
$|U_1|\geq\delta(G)-|T_1|\overset{(\ref{d37})}{>}\frac{m+3n-1}{4}\geq n$ since $m\geq n+1$,
a contradiction to inequality (\ref{d41}). Thus $|R'_2|\leq n-1$.

(iii) Note that $|R'_2\uplus V''_2|\overset{(\ref{d35})}{=}|V_2\backslash T_2|\overset{(\ref{d3})}{\geq}n$. Combining with (ii), $V''_2\neq\emptyset$.

Suppose that $V''_1=\emptyset$. Now $V_1\overset{(\ref{d35})}{=}R_1$, then $|R'_1|=|R_1\backslash T_1|=|V_1\backslash T_1|\overset{(\ref{d37})}{\geq}\frac{m-1}{2}+n$. By Claim \ref{d0}(ii),
$G[R'_1, R'_2\uplus V''_2]\subseteq G_B$. Combining (i) and Proposition \ref{pr2}(i), the blue component $\mathcal{H}$ contains $R'_2\uplus V''_2$. Now $(R'_2\uplus V''_2)\backslash U_2\neq\emptyset$, otherwise $|U_2|\geq|R'_2\uplus V''_2|\overset{(\ref{d35})}{=}|V_2\backslash T_2|\overset{(\ref{d3})}{\geq}n$, contradicting to inequality (\ref{d41}). Let $x\in(R'_2\cup V''_2)\backslash U_2$. By Claim \ref{d0}(ii)-(iii), $N_R(x)\subseteq T_1$ since $V''_1=\emptyset$ and $B_1\subseteq T_1$. Since $U$ is a minimum vertex of $\mathcal{H}$, $N_B(x)\subseteq U_1$. Then $|U_1\cup T_1|\geq\delta(G)$, and so $|U_1|\geq\delta(G)-|T_1|\overset{(\ref{d37})}{>}\frac{m+3n-1}{4}\geq n$ since $m\geq n+1$, a contradiction to inequality (\ref{d41}). Thus $V''_1\neq\emptyset$.
\end{proof}

Let $i\in[2]$. Since $B_i\subseteq R_i$, $G[B_i, V''_{3-i}]=\emptyset$ by Claim \ref{d0}(vi). By Claim \ref{11}(iii) and Proposition \ref{pr1}, we have that
\begin{equation}\label{d42}
|B_i|<\frac{m+n-1}{4}.
\end{equation}
Combining with inequality (\ref{d1}), we have that $m\geq3n+2$.
%\begin{equation}\label{d43}
%m\geq3n+2.
%\end{equation}
Now $|T_2|>\frac{m+n-1}{2}$, otherwise $|R'_2|\overset{(\ref{d35})}{=}|V_2\backslash V''_2|-|T_2|
\overset{(\ref{d8})}{>}\frac{3}{4}(m+n-1)-|T_2|>\frac{m+n-1}{4}\geq n$ since $m\geq3n+2$, a contradiction to Claim \ref{11}(ii). Then
\begin{equation}\label{d44}
|T_1|\overset{(\ref{d3})}{\leq}m-1-|T_2|<\frac{m-n-1}{2}.
\end{equation}

\begin{claim}\label{JT10}  For each $i\in[2]$, the following holds.

{\rm (i)} $V''_i$ is contained in some blue component of $G$, say $\mathcal{F}_i$.

{\rm (ii)} $|\mathcal{F}_i\cap V_{3-i}|>|\mathcal{B}|$.

{\rm (iii)} $|V''_i|\leq n-1$.
\end{claim}
\begin{proof} (i) Note that
$|R'_{3-i}\uplus T'_{3-i}|\overset{(\ref{d35})}{=}|V_{3-i}\backslash(B_{3-i}\uplus V''_{3-i})|
\overset{(\ref{d8})}{\underset{(\ref{d42})}{>}}\frac{m+n-1}{2}$. By Claim \ref{d0}(iii), $G[V''_i, R'_{3-i}\uplus T'_{3-i}]\subseteq G_B$. Combining with Proposition \ref{pr2}(i), $V''_i$ is contained in some blue component of $G$, say $\mathcal{F}_i$.

(ii) By (i), $V''_i\subseteq\mathcal{F}_i\cap V_i$. Let $x\in V''_i$, by Claim \ref{d0}(iii), $N_R(x)\subseteq V''_{3-i}$ since $B_{3-i}\subseteq T_{3-i}$ and $N_B(x)\subseteq T'_{3-i}\uplus R'_{3-i}\uplus V''_{3-i}$, then
$|\mathcal{F}_i\cap V_{3-i}|\geq|N_B(x)\cap(T'_{3-i}\uplus R'_{3-i})|\geq\delta(G)-|V''_{3-i}|
\overset{(\ref{d8})}{>}\frac{m+n-1}{2}\overset{(\ref{d42})}{>}|B_1|+|B_2|=|\mathcal{B}|.$

(iii) If $|V''_i|\geq n$ for some $i\in[2]$, then by (i)-(ii), $|\mathcal{F}_i\cap V_i|\geq |V''_i|\geq n$ and $|\mathcal{F}_i\cap V_{3-i}|>|\mathcal{B}|$, contradicting to the maximality of $\mathcal{B}$. Thus for each $i\in[2]$, $|V''_i|\leq n-1$.
\end{proof}

By Claim \ref{JT10}(iii), $|R'_2|\overset{(\ref{d35})}{=}|V_2\backslash(T_2\uplus V''_2)|\overset{(\ref{d3})}{\geq}n-|V''_2|\geq1$. Let $x\in R'_2$. By Claim \ref{d0}(ii), $N_R(x)\subseteq T_1$ and $N_B(x)\subseteq T'_1\uplus R'_1\uplus V''_1$. Then by Claim \ref{11}(i), $|\mathcal{H}\cap V_1|\geq|N_B(x)|\geq\delta(G)-|T_1|\overset{(\ref{d44})}{>}\frac{m+5n-1}{4}\geq\frac{3n}{2}$ since $m\geq n+1$. Suppose that $G_B[V''_1, R'_2]\neq\emptyset$. By Claim \ref{11}(i) and Claim \ref{JT10}(i), $\mathcal{H}=\mathcal{F}_1$. By Claim \ref{JT10}(ii), $|\mathcal{H}\cap V_2|=|\mathcal{F}_1\cap V_2|>|\mathcal{B}|$, a contradiction to the maximality of $\mathcal{B}$.

Suppose that $G_B[V''_1, R'_2]=\emptyset$, then $G[V''_1, R'_2]=\emptyset$ by Claim \ref{d0}(ii). Since $B_2\subseteq T_2$, $G[V''_1, B_2]=\emptyset$ by Claim \ref{d0}(vi). Thus $G[V''_1, B_2\uplus R'_2]=\emptyset$. Combining Claim \ref{11}(iii) and Proposition \ref{pr1}, $|B_2\uplus R'_2|<\frac{m+n-1}{4}$. By Claim \ref{JT10}(iii),
$|T'_2|\overset{(\ref{d35})}{=}|V_2\backslash V''_2|-|B_2\uplus R'_2|\geq m-|B_2\uplus R'_2|>\frac{3m-n+1}{4}$. Since $B_2\subseteq T_2$,
$|T_2|=|B_2|+|T'_2|\overset{(\ref{d1})}{\geq}n+|T'_2|>\frac{3m+3n+1}{4}$,
and so $|T_1|\overset{(\ref{d3})}{\leq}m-1-|T_2|<\frac{m-3n-5}{4}$. Now $|R'_1|\overset{(\ref{d35})}{=}|V_1\backslash(T_1\uplus V''_1)|
\overset{(\ref{d8})}{>}\frac{3}{4}(m+n-1)-|T_1|>\frac{m+3n+1}{2}$.
By Claim \ref{d0}(ii), $G[R'_1, R'_2\uplus V''_2]\subseteq G_B$.
Combining Proposition \ref{pr2}(i), Claim \ref{11}(i) and Claim \ref{JT10}(i), $\mathcal{H}=\mathcal{F}_2$. By Claim \ref{JT10}(ii),
$|\mathcal{H}\cap V_1|=|\mathcal{F}_2\cap V_1|>|\mathcal{B}|$ and $|\mathcal{H}\cap V_2|\geq|R'_2\uplus V''_2|\overset{(\ref{d35})}{=}|V_1\backslash T_2|\overset{(\ref{d3})}{\geq}n$, a contradiction to the maximality of $\mathcal{B}$.\hfill$\Box$

\vspace{8pt}
\noindent{\bf Proof of Theorem \ref{Sta}.} Set $n_0=1+\frac{1}{\gamma}$, then we have that $3n>m>n\geq1+\frac{1}{\gamma}$. Suppose that for some red-blue-edge-coloring of $G$ which is not $\gamma$-missing, there exists neither a red connected matching of size $(1+\gamma)m$, nor a blue connected matching of size $(1+\gamma)n$. If this edge coloring is a $\gamma$-coloring, then we are done. Thus we assume that this edge coloring is not a $\gamma$-coloring.
By Lemma \ref{3}, we can assume that $\mathcal{B}$ is a largest blue component such that $|\mathcal{B}\cap V_i|\geq(1+\gamma)n$ for each $i\in[2]$. Let $S$ be a minimum vertex cover of $\mathcal{B}$. By Lemma \ref{4}, $\mathcal{B}\backslash S$ is contained in some red component of $G$, say $\mathcal{R}$. Let $T$ be a minimum vertex cover of $\mathcal{R}$. For each $i\in[2]$, let $B_i=\mathcal{B}\cap V_i$, $S_i=S\cap V_i$, $B'_i=B_i\backslash S_i$ and $V'_i=V_i\backslash B_i$. For each $i\in[2]$, let $R_i=\mathcal{R}\cap V_i$, $T_i=T\cap V_i$, $T'_i=T_i\backslash B_i$, $R'_i=R_i\backslash(B_i\cup T_i)$, and $V''_i=V_i\backslash(B_i\cup R_i)$. Then for each $i\in[2]$, $B'_i\subseteq B_i\cap R_i$ and $V'_i=R'_i\uplus T'_i\uplus V''_i$. Let $N:=m+n-1$, then $\delta(G)>(\frac{3}{4}+\gamma)(m+n-1)=(\frac{3}{4}+\gamma)N$. Since $3n>m>n\geq1+\frac{1}{\gamma}$,
\begin{equation}\label{ssa1}
\delta(G)>(\frac{3}{4}+\gamma)N>(1+\gamma)m.
\end{equation}

\begin{claim}\label{s0} For each $i\in[2]$, the following holds.

{\rm (i)} If $x\in B'_i\backslash T_i$, then $N_B(x)\subseteq S_{3-i}$ and $N_R(x)\subseteq T_{3-i}$.

{\rm (ii)} If $x\in R'_i$, then $N_B(x)\subseteq T'_{3-i}\uplus R'_{3-i}\uplus V''_{3-i}$ and $N_R(x)\subseteq T_{3-i}$.

{\rm (iii)} If $x\in V''_i$, then $N_B(x)\subseteq T'_{3-i}\uplus R'_{3-i}\uplus V''_{3-i}$ and $N_R(x)\subseteq(S_{3-i}\backslash R_{3-i})\uplus V''_{3-i}$.

{\rm (iv)} If $x\in S_i\backslash R_i$, then $N_B(x)\subseteq B_{3-i}$ and $N_R(x)\subseteq(S_{3-i}\backslash R_{3-i})\uplus V''_{3-i}$.

{\rm (v)} If $x\in (S_i\cap R_i)\backslash T_i$, then $N_B(x)\subseteq B_{3-i}$ and $N_R(x)\subseteq T_{3-i}$.

{\rm (vi)} $G[V''_i, B_{3-i}\cap R_{3-i}]=\emptyset$.
%{\rm (vii)} $G[S_i\backslash R_i, T'_{3-i}\uplus R'_{3-i}]=\emptyset$.
\end{claim}
\begin{proof} (i) Let $x\in B'_i\backslash T_i\subseteq(B_i\cap R_i)\backslash T_i$. Since $T$ is a minimum vertex cover of $\mathcal{R}$, $N_R(x)\subseteq T_{3-i}$. Since $S$ is a minimum vertex cover of $\mathcal{B}$, $N_B(x)\subseteq S_{3-i}$.

(ii) Let $x\in R'_i=R_i\backslash(B_i\cup T_i)$. If $N_B(x)\cap B_{3-i}\neq\emptyset$, then since $\mathcal{B}$ is a largest blue component, $x\in B_i$, a contradiction. Thus $N_B(x)\subseteq V_{3-i}\backslash B_{3-i}=T'_{3-i}\uplus R'_{3-i}\uplus V''_{3-i}$. Since $T$ is a minimum vertex cover of $\mathcal{R}$, $N_R(x)\subseteq T_{3-i}$.

(iii) Let $x\in V''_i=V_i\backslash(B_i\cup R_i)$. If $N_B(x)\cap B_{3-i}\neq\emptyset$, then since $\mathcal{B}$ is a largest blue component, $x\in B_i$, a contradiction. Thus $N_B(x)\subseteq V_{3-i}\backslash B_{3-i}=T'_{3-i}\uplus R'_{3-i}\uplus V''_{3-i}$. If $N_R(x)\cap R_{3-i}\neq\emptyset$, then since $\mathcal{R}$ is a largest red component, $x\in R_i$, a contradiction. Thus $N_R(x)\subseteq V_{3-i}\backslash R_{3-i}=(S_{3-i}\backslash R_{3-i})\uplus V''_{3-i}$.

(iv) Let $x\in S_i\backslash R_i\subseteq B_i\backslash R_i$. Since $\mathcal{B}$ is a largest blue component, $N_B(x)\subseteq B_{3-i}$. If $N_R(x)\cap R_{3-i}\neq\emptyset$, then since $\mathcal{R}$ is a largest red component, $x\in R_i$, a contradiction. Thus $N_R(x)\subseteq V_{3-i}\backslash R_{3-i}=(S_{3-i}\backslash R_{3-i})\uplus V''_{3-i}$.

(v) Let $x\in (S_i\cap R_i)\backslash T_i$. Since $\mathcal{B}$ is a largest blue component, $N_B(x)\subseteq B_{3-i}$. Since $T$ is a minimum vertex cover of $\mathcal{R}$, $N_R(x)\subseteq T_{3-i}$.

(vi) If $V''_i=\emptyset$ or $B_{3-i}\cap R_{3-i}=\emptyset$, then we are done. Suppose that $V''_i\neq\emptyset$ and $B_{3-i}\cap R_{3-i}\neq\emptyset$. For any $x\in V''_{3-i}$, by (iii), $N_B(x)\subseteq V_{3-i}\backslash B_{3-i}$ and $N_R(x)\subseteq V_{3-i}\backslash R_{3-i}$, then $N_G(x)\cap(B_{3-i}\cap R_{3-i})=\emptyset$. Thus $G[V''_i, B_{3-i}\cap R_{3-i}]=\emptyset$.
%(vii) If $S_i\backslash R_i=\emptyset$ or $T'_{3-i}\uplus R'_{3-i}=\emptyset$, then we are done. Suppose that $S_i\backslash R_i\neq\emptyset$ and $T'_{3-i}\uplus R'_{3-i}\neq\emptyset$. For any $x\in S_i\backslash R_i$, by (iv), $N_B(x)\subseteq B_{3-i}$ and $N_R(x)\subseteq(S_{3-i}\backslash R_{3-i})\uplus V''_{3-i}$, then $N_G(x)\cap(T'_{3-i}\uplus R'_{3-i})=\emptyset$. Thus $G[S_i\backslash R_i, T'_{3-i}\uplus R'_{3-i}]=\emptyset$.
\end{proof}

By Theorem \ref{EK},
\begin{equation}\label{sa2}
|S_1|+|S_2|=|S|=\alpha'(\mathcal{B})<(1+\gamma)n,
\end{equation}
and
\begin{equation}\label{sa3}
|T_1|+|T_2|=|T|=\alpha'(\mathcal{R})<(1+\gamma)m.
\end{equation}
Then
\begin{equation}\label{sa4}
|S\cup T|\leq|S|+|T|<(1+\gamma)(m+n).
\end{equation}
For each $i\in[2]$, by the hypothesis, we have that
\begin{equation}\label{sa1}
|B_i|=|S_i|+|B'_i|\geq(1+\gamma)n,
\end{equation}
then $B'_i\neq\emptyset$. Since $B'_i\subseteq B_i\cap R_i$, by Claim \ref{s0}(vi) and Fact \ref{pr1}, we have that
\begin{equation}\label{sa6}
|V''_i|<(\frac{1}{4}-\gamma)N.
\end{equation}

For each $i\in[2]$, if $x\in B'_i\backslash T_i$, by Claim \ref{s0}(i), $N_G(x)\subseteq S_{3-i}\cup T_{3-i}$, then $|S_{3-i}\cup T_{3-i}|\geq\delta(G)$. If $B'_i\backslash T_i\neq\emptyset$ for each $i\in[2]$, then $|S\cup T|=|S_1\cup T_1|+|S_2\cup T_2|\geq2\delta(G)>(\frac{3}{2}+2\gamma)N$, a contradiction to inequality (\ref{sa4}). Thus either $B'_1\subseteq T_1$ or $B'_2\subseteq T_2$. Without loss of generality, assume that
\begin{equation}\label{sa5}
B'_1\subseteq T_1.
\end{equation}
Let $S'_2=S_2\backslash R_2$, $SR_2=(S_2\cap R_2)\backslash T_2$, $ST_2=S_2\cap T_2$ and $B''_2=B'_2\backslash T_2$. Now
\begin{equation}\label{sa10}
V_1=(S_1\backslash T_1)\uplus T_1\uplus R'_1\uplus V''_1.
\end{equation}
and
\begin{equation}\label{sa11}
V_2=S'_2\uplus SR_2\uplus B''_2\uplus T_2\uplus R'_2\uplus V''_2.
\end{equation}
%Let $i\in[2]$. For any $x\in R'_i$, $N_B(x)\subseteq T'_{3-i}\uplus R'_{3-i}\uplus V''_{3-i}$ and $N_R(x)\subseteq T_{3-i}$ by Claim \ref{s0}(ii), then
%\begin{equation}\label{sa7}
%|N_B(x)\cap(R'_{3-i}\uplus V''_{3-i})|\geq\delta(G)-|T_{3-i}|,
%\end{equation}
%and so
%\begin{equation}\label{sa8}
%|N_B(x)\cap R'_{3-i}|\geq\delta(G)-|T_{3-i}\uplus V''_{3-i}|.
%\end{equation}
%For any $x\in V''_1$, $N_B(y)\subseteq T'_2\uplus R'_2\uplus V''_2$ and $N_R(y)\subseteq S'_2\uplus V''_2$ by Claim \ref{s0}(iii), then
%\begin{equation}\label{sa9}
%|N_B(y)\cap R'_2|\geq\delta(G)-|S'_2\uplus T'_2\uplus V''_2|.
%\end{equation}

For any $x\in R'_1$, by Claim \ref{s0}(ii), $N_G(x)\subseteq T'_2\uplus R'_2\uplus V''_2$, then $N_G(x)\cap(S'_2\uplus SR_2\uplus B''_2)=\emptyset$ by (\ref{sa11}). Thus $G[R'_1, S'_2\uplus SR_2\uplus B''_2]=\emptyset$. Since $B_2\cap R_2=ST_2\uplus SR_2\uplus B'_2$, by Claim \ref{s0}(vi), $G[V''_1, ST_2\uplus SR_2\uplus B'_2]=\emptyset$. Thus
\begin{equation}\label{sa12}
G[R'_1\uplus V''_1, SR_2\uplus B''_2]=\emptyset.
\end{equation}

Next we split our argument into two cases.

{\bf Case 1.} $B'_1\subseteq T_1$ and $B'_2\backslash T_2\neq\emptyset$.

In this case, $B''_2=B'_2\backslash T_2\neq\emptyset$. For any $x\in B''_2$, $N_G(x)\subseteq S_1\cup T_1$ by Claim \ref{s0}(i). Then
\begin{equation}\label{sa16}
|S_1\cup T_1|\geq\delta(G)>(\frac{3}{4}+\gamma)N.
\end{equation}

\begin{claim}\label{sc2}
The following holds.

{\rm (i)} $R'_1\uplus V''_1=\emptyset$. Furthermore, $V_1=S_1\cup T_1$.

{\rm (ii)} $R'_2=\emptyset$.
\end{claim}
\begin{proof} (i) Suppose that $R'_1\uplus V''_1\neq\emptyset$. By (\ref{sa12}) and Fact \ref{pr1}, $|SR_2\uplus B''_2|<(\frac{1}{4}-\gamma)N$. Note that
$|S_2\cup T_2|=|S\cup T|-|S_1\cup T_1|\underset{(\ref{sa16})}{\overset{(\ref{sa4})}{<}}
(1+\gamma)(m+n)-(\frac{3}{4}+\gamma)N=\frac{m+n+3}{4}+\gamma$. Then
$|B_2\cup T_2|=|S'_2\uplus SR_2\uplus B''_2\uplus T_2|\leq|S_2\cup T_2|+|SR_2\uplus B''_2|
<\frac{m+n+3}{4}+\gamma+(\frac{1}{4}-\gamma)N=(\frac{1}{2}-\gamma)(m+n)+2\gamma+\frac{1}{2}$, and so
$|R'_2|\overset{(\ref{sa11})}{=}|V_2\backslash V''_2|-|B_2\cup T_2|\overset{(\ref{sa6})}{>}(\frac{3}{4}+\gamma)N-|B_2\cup T_2|>\frac{N}{4}+2\gamma(m+n)-3\gamma-1>\frac{N}{4}$ since $m>n>1+\frac{1}{\gamma}$.

For any $x\in R'_2$, by Claim \ref{s0}(ii), $N_G(x)\subseteq T_1\uplus R'_1\uplus V''_1$, then $N_G(x)\cap(S_1\backslash T_1)=\emptyset$ by (\ref{sa10}). Thus $G[S_1\backslash T_1, R'_2]=\emptyset$. If $S_1\backslash T_1\neq\emptyset$, then by Fact \ref{pr1}, $|R'_2|<(\frac{1}{4}-\gamma)N$, contradicting that $|R'_2|>\frac{N}{4}$. Thus $S_1\subseteq T_1$, and so  $|T_1|\overset{(\ref{sa16})}{>}(\frac{3}{4}+\gamma)N\overset{(\ref{ssa1})}{>}(1+\gamma)m$, a contradiction to inequality (\ref{sa3}). Thus $R'_1\uplus V''_1=\emptyset$. By (\ref{sa10}), we have that $V_1=S_1\cup T_1$.

(ii) Suppose that $R'_2\neq\emptyset$. Let $x\in R'_2$, by (i) and Claim \ref{s0}(ii), $N_G(x)\subseteq T_1$. Then $|T_1|\geq\delta(G)>(\frac{3}{4}+\gamma)N\overset{(\ref{ssa1})}{>}(1+\gamma)m$, a contradiction to inequality (\ref{sa3}). Thus $R'_2=\emptyset$.
\end{proof}

By Claim \ref{sc2}(i), $|S_2\cup T_2|=|S\cup T|-|S_1\cup T_1|=|S\cup T|-|V_1|\overset{(\ref{sa4})}{<}\gamma(m+n)+1$. By Claim \ref{sc2}(ii), $|B''_2|\overset{(\ref{sa11})}{=}|V_2\backslash V''_2|-|S_2\cup T_2|\overset{(\ref{sa6})}{>}
(\frac{3}{4}+\gamma)N-|S_2\cup T_2|>\frac{3}{4}(m+n-1)-\gamma-1$.
By Claim \ref{sc2}(i), $|T_1\backslash S_1|=|V_1\backslash S_1|\overset{(\ref{sa2})}{>}m-\gamma n-1$ and $|S_1\backslash T_1|=|V_1\backslash T_1|\overset{(\ref{sa3})}{>}n-\gamma m-1$. Combining with Claim \ref{s0}(i), the coloring is a $\gamma$-coloring as witnessed by $B''_2 \subseteq V_2$ and the partition $\{S_1\backslash T_1, S_1\cap T_1, T_1\backslash S_1\}$ of $V_1$.

{\bf Case 2.} For each $i\in[2]$, $B'_i\subseteq T_i$.

For each $i\in[2]$, we have that $B_i\subseteq S_i\cup T_i$. If $x\in S_i\backslash R_i$, by Claim \ref{s0}(iv), $N_G(x)\subseteq B_{3-i}\uplus V''_{3-i}$, then $|B_{3-i}\uplus V''_{3-i}|\geq\delta(G)$, and so
$|S_{3-i}\cup T_{3-i}|\geq|B_{3-i}|\geq\delta(G)-|V''_{3-i}|\overset{(\ref{sa6})}{>}(\frac{1}{2}+2\gamma)N.$
If $x\in(S_i\cap R_i)\backslash T_i$, by Claim \ref{s0}(v), $N_G(x)\subseteq B_{3-i}\cup T_{3-i}$, then $|S_{3-i}\cup T_{3-i}|=|B_{3-i}\cup T_{3-i}|\geq\delta(G)>(\frac{3}{4}+\gamma)N>(\frac{1}{2}+2\gamma)N$ since $\gamma<\frac{1}{4}$.
Suppose that $S_i\backslash T_i=(S_i\backslash R_i)\uplus((S_i\cap R_i)\backslash T_i)\neq\emptyset$ for each $i\in[2]$, then each $|S_i\cup T_i|>(\frac{1}{2}+2\gamma)N$, and so $|S\cup T|=|S_1\cup T_1|+|S_2\cup T_2|>(1+4\gamma)N>(1+\gamma)(m+n)$ since $m>n\geq1+\frac{1}{\gamma}$, a contradiction to inequality (\ref{sa4}).
Thus either $S_1\subseteq T_1$ or $S_2\subseteq T_2$. Without loss of generality, assume that $S_1\subseteq T_1$. Combining with (\ref{sa5}), we have that
\begin{equation}\label{sa13}
B_1\subseteq T_1,
\end{equation}
then
\begin{equation}\label{sa18}
V_1=R_1\uplus V''_1= T'_1\uplus R'_1\uplus V''_1.
\end{equation}

Now we split the remainder into two cases.

{\bf Subcase 2.1.} For each $i\in[2]$, $S_i\subseteq T_i$.

In this case, we have that $B_i\subseteq T_i$ for each $i\in[2]$.

\begin{claim}\label{sc1.1}
For each $i\in[2]$, the following holds.

{\rm (i)} $V''_i=\emptyset$. Furthermore, $V_i=R_i=T_i\uplus R'_i$.

{\rm (ii)} $|R'_i|>(1+\gamma)n$.
\end{claim}
\begin{proof} (i) Let $i\in[2]$. Since $B_i\subseteq T_i$, by Claim \ref{s0}(vi), $G[V''_i, B_{3-i}]=\emptyset$. Now $V''_i=\emptyset$, otherwise by Fact \ref{pr1}, $|B_{3-i}|<(\frac{1}{4}-\gamma)N\leq(1+\gamma)n$ since $3n>m>1+\frac{1}{\gamma}$, a contradiction to inequality (\ref{sa1}). Thus $V_i=R_i=T_i\uplus R'_i$ for each $i\in[2]$.

(ii) Let $i\in[2]$. Since $B_{3-i}\subseteq T_{3-i}$, $|T_{3-i}|\geq|B_{3-i}|\overset{(\ref{sa1})}{\geq}(1+\gamma)n$, and so $|T_i|\overset{(\ref{sa3})}{<}(1+\gamma)m-|T_{3-i}|\leq(1+\gamma)(m-n)$. By (i), $|R'_i|=|V_i\backslash T_i|>(1+\gamma)n+n-\gamma m-1>(1+\gamma)n$ since $3n>m>1+\frac{1}{\gamma}$ and $\gamma<\frac{1}{4}$.
\end{proof}

Without loss of generality, by inequality (\ref{sa3}), we can assume that $|T_1|<\frac{1+\gamma}{2}m$. By Claim \ref{sc1.1}(i), $|R'_1|=|V_1\backslash T_1|>\frac{1-\gamma}{2}m+n-1>(\frac{1}{2}-2\gamma)N$. By Claim \ref{s0}(ii), $G[R'_1, R'_2]\subseteq G_B$. By Fact \ref{pr2}(i), $R'_2$ is contained in a blue component of $G$, say $\mathcal{H}_1$.

For any $x\in R'_1$, by Claim \ref{s0}(ii), $N_R(x)\subseteq T_2$ and $N_B(x)\subseteq T'_2\uplus R'_2$ by Claim \ref{sc1.1}(i), then $|N_B(x)\cap R'_2|\geq\delta(G)-|T_2|\overset{(\ref{sa3})}{>}(\frac{3}{4}+\gamma)N-(1+\gamma)m\overset{(\ref{ssa1})}{>}0$. Since $R'_2\subseteq \mathcal{H}_1\cap V_2$, $R'_1\subseteq \mathcal{H}_1\cap V_1$. For each $i\in[2]$, $|\mathcal{H}_1\cap V_i|\geq|R'_i|>(1+\gamma)n$ by Claim \ref{sc1.1}(ii).

By Claim \ref{sc1.1}(i), $|\mathcal{H}_1|\geq|R'_1|+|R'_2|=|V_1\backslash T_1|+|V_2\backslash T_2|=2N-|T|\overset{(\ref{sa3})}{>}2(n-1)+(1-\gamma)m$. Recall that $B_i\subseteq T_i$ for each $i\in[2]$, then $|\mathcal{B}|=|B_1|+|B_2|\leq|T_1|+|T_2|\overset{(\ref{sa3})}{<}(1+\gamma)m$. Since $3n>m>1+\frac{1}{\gamma}$, $|\mathcal{B}|<(1+\gamma)m<2(n-1)+(1-\gamma)m<|\mathcal{H}_1|$, a contradiction to the maximality of $\mathcal{B}$.

{\bf Subcase 2.2.} $S_1\subseteq T_1$ and $S_2\backslash T_2\neq\emptyset$.

\begin{claim}\label{sc1.2}
The following holds.

{\rm (i)} $V''_2=\emptyset$.

{\rm (ii)} $SR_2\uplus B''_2=\emptyset$. Furthermore, $V_2=S'_2\uplus T_2\uplus R'_2$.
\end{claim}
\begin{proof} (i) By (\ref{sa13}) and Claim \ref{s0}(vi), $G[B_1, V''_2]=\emptyset$. Now $V''_2=\emptyset$, otherwise by Fact \ref{pr1}, $|B_1|<(\frac{1}{4}-\gamma)N<(1+\gamma)n$ since $3n>m>1+\frac{1}{\gamma}$, contradicting to inequality (\ref{sa1}).

(ii) Note that $|R'_1\uplus V''_1|\overset{(\ref{sa18})}{=}|V_1\backslash T_1|\overset{(\ref{sa3})}{>}n-\gamma m-1>(\frac{1}{4}-\gamma)N$ since $3n>m>1+\frac{1}{\gamma}$. Then $SR_2\uplus B''_2=\emptyset$, otherwise by (\ref{sa12}) and Fact \ref{pr1}, $|R'_1\uplus V''_1|<(\frac{1}{4}-\gamma)N$, a contradiction. Combining with (i) and (\ref{sa11}), we have that $V_2=S'_2\uplus T_2\uplus R'_2$.
\end{proof}

Note that $|R_1|\overset{(\ref{sa18})}{=}|V_1\backslash V''_1|\overset{(\ref{sa6})}{>}(\frac{3}{4}+\gamma)N
\overset{(\ref{ssa1})}{>}(1+\gamma)m\overset{(\ref{sa3})}{>}|T_1|$.
Thus $R'_1=R_1\backslash T_1\neq\emptyset$. For any $x\in R'_1$, by Claim \ref{s0}(ii) and Claim \ref{sc1.2}(i), $N_G(x)\subseteq R_2$. Thus $|R_2|\geq\delta(G)>(\frac{3}{4}+\gamma)N$. By Claim \ref{sc1.2}(ii),
\begin{equation}\label{sa15}
|S'_2|=|V_2|-|R_2|<(\frac{1}{4}-\gamma)N.
\end{equation}
By Claim \ref{sc1.2}(ii), $S'_2=S_2\backslash R_2=S_2\backslash T_2\neq\emptyset$. For any $x\in S'_2$, by Claim \ref{s0}(iv) and (\ref{sa13}), $N_G(x)\subseteq B_1\uplus V''_1$. Then $|B_1\uplus V''_1|\geq\delta(G)$, and so $|T_1|\overset{(\ref{sa13})}{\geq}|B_1|\geq\delta(G)-|V''_1|
\overset{(\ref{sa6})}{>}(\frac{1}{2}+2\gamma)N$. Then $|T_2|\overset{(\ref{sa3})}{<}(1+\gamma)m-|T_1|<(1+\gamma)m-(\frac{1}{2}+2\gamma)N$, and so
\begin{equation}\label{sa14}
|R'_2|=|R_2\backslash T_2|>(\frac{3}{4}+\gamma)N-|T_2|
>(\frac{5}{4}+3\gamma)(n-1)+(\frac{1}{4}+2\gamma)m.
\end{equation}
Since $3n>m>1+\frac{1}{\gamma}$, $|R'_2|\overset{(\ref{sa14})}{>}(\frac{1}{2}-2\gamma)N$. By Claim \ref{s0}(ii), $G[R'_1\uplus V''_1, R'_2]\subseteq G_B$. By Fact \ref{pr2}(i), $R'_1\uplus V''_1$ is contained in a blue component of $G$, say $\mathcal{H}_2$.

For any $x\in R'_2$, by Claim \ref{s0}(ii), $N_B(x)\subseteq T'_1\uplus R'_1\uplus V''_1$ and $N_R(x)\subseteq T_1$, then $|N_B(x)\cap(R'_1\uplus V''_1)|\geq\delta(G)-|T_1|
\overset{(\ref{sa3})}{>}(\frac{3}{4}+\gamma)N-(1+\gamma)m\overset{(\ref{ssa1})}{>}0$. Since $R'_1\uplus V''_1\subseteq \mathcal{H}_2\cap V_1$, $R'_2\subseteq \mathcal{H}_2\cap V_2$.
Since $m>n\geq1+\frac{1}{\gamma}$, $|\mathcal{H}_2\cap V_2|\geq|R'_2|\overset{(\ref{sa14})}{>}(1+\gamma)n$. By Claim \ref{sc1.2}(ii),
$|T_2|\geq|ST_2\uplus B'_2|=|B_2\backslash S'_2|
\underset{(\ref{sa15})}{\overset{(\ref{sa1})}{>}}(1+\gamma)n-(\frac{1}{4}-\gamma)N$, then $|T_1|\overset{(\ref{sa3})}{<}(1+\gamma)m-|T_2|<(1+\gamma)(m-n)+(\frac{1}{4}-\gamma)N$. Then
$|\mathcal{H}_2\cap V_1|\geq|R'_1\uplus V''_1|\overset{(\ref{sa18})}{=}|V_1\backslash T_1|
>(1+\gamma)n+(\frac{3}{4}+\gamma)N-(1+\gamma)m\overset{(\ref{ssa1})}{>}(1+\gamma)n$.
Now $|\mathcal{H}_2\cap V_i|>(1+\gamma)n$ for each $i\in[2]$.

By Claim \ref{sc1.2}(ii), $|\mathcal{H}_2|\geq|R'_1\uplus V''_1|+|R'_2|\overset{(\ref{sa18})}{=}|V_1\backslash T_1|+|V_2\backslash(S'_2\uplus T_2)|=2N-|T|-|S'_2|\underset{(\ref{sa15})}{\overset{(\ref{sa3})}{>}}
(\frac{7}{4}+\gamma)(n-1)+\frac{3}{4}m$ and
$|\mathcal{B}|=|B_1|+|B_2|\overset{(\ref{sa13})}{\leq}|T_1|+|S'_2\uplus T_2|=|S'_2|+|T|
\underset{(\ref{sa15})}{\overset{(\ref{sa3})}{<}}(\frac{1}{4}-\gamma)(n-1)+\frac{5}{4}m$. Since $3n>m>1+\frac{1}{\gamma}$, $|\mathcal{B}|<(\frac{1}{4}-\gamma)(n-1)+\frac{5}{4}m<
(\frac{7}{4}+\gamma)(n-1)+\frac{3}{4}m<|\mathcal{H}_2|$, a contradiction to the maximality of $\mathcal{B}$.
\hfill$\Box$

\section{Monochromatic Cycles}

\ For completeness, we will explain how to expand the large monochromatic connected matchings in the auxiliary graph into the monochromatic cycles in the initial graph in this section. The method was initially introduced by {\L}uczak \cite{Luc1999}.

Given a graph $G$, let $X$ and $Y$ be disjoint subsets of $V(G)$. The {\em density} of the pair $(X, Y)$ is the value $d(X, Y)=\dfrac{e(G[X, Y])}{|X||Y|}$. For $\epsilon>0$, the pair $(X, Y)$ is called {\em $\epsilon$-regular} for $G$ if $|d(X, Y)-d(X', Y')|<\epsilon$ for any $X'\subseteq X$ and $Y'\subseteq Y$ with $|X'|>\epsilon|X|$ and $|Y'|>\epsilon|Y|$.

\begin{fact}\label{fact}
Let $(U, V)$ be an $\epsilon$-regular pair with density d and $V'\subseteq V$ with $|V'|>\epsilon|V|$, then all but at most $\epsilon|U|$ vertices $u\in U$ satisfying $|N(u)\cap V'|>(d-\epsilon)|V'|$.
\end{fact}

Let $G$ be a graph and $c$ be a 2-edge-coloring of $G$. For a subgraph $G'$ of $G$, the edge-coloring $c$ restricted to $E(G')$ is called an induced 2-edge-coloring of $G'$.

We will use the following bipartite degree form for 2-colored regularity lemma adapted to our needs.

\begin{lemma}[2-colored Regularity Lemma-Bipartite Degree Form \cite{Sze1976}]\label{reg} For any $\epsilon>0$ and positive integer $k_0$, there exists an $M=M(\epsilon, k_0$) such that for any $2$-edge-colored balanced bipartite graph $G[X, Y]$ on order $2N\geq M$ and any $d\in[0,1]$, there exists an integer $k$, a partition $\{X_0, X_1, \ldots, X_k\}$ of $X$, a partition $\{Y_0, Y_1, \ldots, Y_k\}$ of $Y$, and a subgraph $G'\subseteq G$ with the following properties:

{\rm(i)} $|X_0|=|Y_0|\leq\epsilon N$.

{\rm(ii)} $k_0\leq k\leq M$.

{\rm(iii)} For any $1\leq i, j\leq k$, $|X_i|=|Y_j|=n$.

{\rm(iv)} For any $v\in V(G)$, $d_{G'}(v)>d_{G}(v)-(2d+\epsilon)N$.

{\rm(v)} For any $1\leq i, j\leq k$, the pair $(X_i, Y_j)$ is $\epsilon$-regular for $G'_R$ with density either $0$ or greater than $d$, and $\epsilon$-regular for $G'_B$ with density either $0$ or greater than $d$, where $E(G')=E(G'_R)\cup E(G'_B)$ is the induced $2$-edge-coloring of $G'$.
\end{lemma}

\begin{defi}[($\epsilon$, d)-reduced graph]\label{reduced} Given a bipartite graph $G[X, Y]$, a partition $\{X_0, X_1, \ldots, X_k\}$ of $X$ and a partition $\{Y_0, Y_1, \ldots, Y_k\}$ of $Y$ satisfying properties {\rm(i)-(v)} of Lemma \ref{reg}, we define the $(\epsilon, d)$-reduced $2$-colored bipartite graph $\Gamma$ on vertex set $\{x_i: i\in[k]\}\uplus\{y_j: j\in[k]\}$ as follows. For any $1\leq i, j\leq k$,

$\bullet$ let $x_i y_j$ be a red edge of $\Gamma$ when $G'_{R}[X_i, Y_j]$ has density at least $d$;

$\bullet$ let $x_i y_j$ be a blue edge of $\Gamma$ when $G'_{B}[X_i, Y_j]$ has density at least $d$.
\end{defi}

The next lemma [\cite{BLS2012}, Lemma 2.2] due to Benevides, {\L}uczak, Scott, Skokan and White guarantees a long monochromatic path in a regular pair.

\begin{lemma}[Benevides, {\L}uczak, Scott, Skokan and White \cite{BLS2012}]\label{path}
For every $0<\beta<1$, there is an $m_0(\beta)$ such that for every $m>m_0(\beta)$ the following holds: Let $G$ be a graph, and let $V_1$, $V_2$ be disjoint subsets of $V(G)$ such that $|V_1|, |V_2|\geq m$. Furthermore let the pair $(V_1, V_2)$ be $\epsilon$-regular for G with density at least $\frac{\beta}{4}$ for some $0<\epsilon<\frac{\beta}{4}$. Then for every pair of vertices $v_1\in V_1$, $v_2\in V_2$ satisfying $|N_G(v_1)\cap V_2|, |N_G(v_2)\cap V_1|\geq\frac{\beta m}{5}$, and for every $1\leq l\leq m-\frac{5\epsilon m}{\beta}$, $G$ contains a path of length $2l+1$ connecting $v_1$ and $v_2$.
\end{lemma}

\noindent{\bf Proof of Theorem \ref{0}.} Assume that $0<\eta<\frac{1}{1000}$ and $N$ is large enough. Let $G[X, Y]$ be a balanced bipartite graph on $2(N-1)$ vertices with $\delta(G)\geq(\frac{3}{4}+3\eta)(N-1)$.

Let $\epsilon=\eta^3$ and $d=\eta$. By Lemma \ref{reg}, there exists a partition $\{U^{(1)}_0, U^{(1)}_1, \ldots, U^{(1)}_{k-1}\}$ of $X$, a partition $\{U^{(2)}_0, U^{(2)}_1, \ldots, U^{(2)}_{k-1}\}$ of $Y$ and a subgraph $G^{'}\subseteq G$ satisfying properties (i)-(v) in Lemma \ref{reg}. Let $\Gamma$ be an $(\epsilon, d)$-reduced 2-colored bipartite graph deduced from $G$ with bipartition $\{u^{(1)}_i: i\in[k-1]\}\uplus\{u^{(2)}_i: i\in[k-1]\}$.

By Lemma \ref{reg}(iv), $\delta(G')>\delta(G)-(2d+\epsilon)(N-1)\geq(\frac{3}{4}+\eta-\eta^3)(N-1)>\frac{3}{4}(N-1)$. For any $1\leq i, j\leq k-1$, by Lemma \ref{reg}(i) and (iii), $(1-\epsilon)\frac{N-1}{k-1}\leq|U^{(1)}_i|=|U^{(2)}_j|=n\leq\frac{N-1}{k-1}$. For any $1\leq i, j\leq k-1$, by Lemma \ref{reg}(v) and Definition \ref{reduced}, $u^{(1)}_iu^{(2)}_j\in E(\Gamma)$ if and only if $G'[U^{(1)}_i, U^{(2)}_j]\neq\emptyset$. Then $\delta(\Gamma)\geq\frac{\delta(G')}{n}>\frac{3}{4}(k-1)$. Thus $\Gamma$ is a 2-colored balanced bipartite graph on 2$(k-1)$ vertices with $\delta(\Gamma)>\frac{3}{4}(k-1)$. By Theorem \ref{1}, each 2-edge-coloring of $\Gamma$ yields a red connected $\lfloor\alpha_1k\rfloor$-matching or a blue connected $\lfloor\alpha_2k\rfloor$-matching.

Suppose that $\Gamma$ contains a red connected $t$-matching $M^*$, where $1\leq t\leq\alpha_1k$. Let $F^*$ be a red minimal tree containing $M^*$. Let $W=u^{(r)}_{i_1}u^{(3-r)}_{i_2}u^{(r)}_{i_3}\cdots u^{(3-r)}_{i_s}u^{(r)}_{i_1}$ be a closed walk in $F^*$ containing $M^*$, then $s\geq2t$. Since $F^*$ is a tree, $W$ must be of even length $s$. Now we view an edge $u^{(p)}_{i_q}u^{(3-p)}_{i_q+1}$ of $W$ as in $M^*$ only when it is an edge in $M^*$ and first appearances in $W$, where $p\in\{1, 2\}$, $q\in[s]$, $i_0=i_{s}$ and $i_{s+1}=i_{1}$.

Applying Fact \ref{fact} repeatedly, for any $q\in[s]$, there exists a vertex $v^{(p)}_{i_q}\in U^{(p)}_{i_q}$, where $p\in\{1, 2\}$, $i_0=i_{s}$ and $i_{s+1}=i_{1}$, such that:

\noindent (i) $v^{(p)}_{i_q}$ has at least $(d-\epsilon)n=(\eta-\eta^3)n\geq\frac{4\eta}{5}n$ red neighbours in both $U^{(3-p)}_{i_{q-1}}$ and $U^{(3-p)}_{i_{q+1}}$;

\noindent (ii) If an edge $u^{(p)}_{i_q}u^{(3-p)}_{i_q+1}$ of $W$ is not in $M^*$, then $v^{(p)}_{i_q}v^{(3-p)}_{i_q+1}$ is a red edge in $G$.

Let $m=(1-\epsilon)\frac{N-1}{k-1}$ and $\beta=4\eta$. By Lemma \ref{path}, we have that for any $1\leq l\leq(1-\frac{5\eta^2}{4})m$, each edge $u^{(p)}_{i_q}u^{(3-p)}_{i_q+1}$ in $M^*$ can be extended a red path of length $2l+1$ connecting vertices $v^{(p)}_{i_q}\in U^{(p)}_{i_q}$ and $v^{(3-p)}_{i_q+1}\in U^{(3-p)}_{i_q+1}$ in $G$. Then there exists a red cycle of each even length $\sum_{j=1}^{t}2l_j+s$, where $1\leq l_j\leq(1-\frac{5\eta^2}{4})m$ for each $j\in[t]$. Let $t=1$, $s=2$, and $l=1$, then there exists a red cycle of length 4. Recall that $N\geq k$. For each $j\in[t]$, let $t=\lfloor\alpha_1k\rfloor$ and $l_j=(1-\frac{5\eta^2}{4})m$. Then
$$\begin{aligned}
\sum_{j=1}^{t}2l_j+s
&=2\sum_{j=1}^{t}l_j+s=2t(1-\frac{5\eta^2}{4})m+s \\
&=2t(1-\frac{5\eta^2}{4})(1-\epsilon)\frac{N-1}{k-1}+s \\
&=2t(1-\frac{5\eta^2}{4})(1-\eta^3)\frac{N-1}{k-1}+s  \\
&>2t(1-\frac{3}{2}\eta^2)\frac{N}{k}+2t\geq(2-3\eta^2)\alpha_1N.
\end{aligned}$$
Therefore there exist red even cycles of each length in $\{4, 6, 8, \ldots, (2-3\eta^2)\alpha_1N\}$.

Suppose that $\Gamma$ contains a blue connected $t$-matching, where $1\leq t\leq\alpha_2k$, then as the same argument above, we have that there exist blue even cycles of each length $\{4, 6, 8, \ldots, (2-3\eta^2)\alpha_2N\}$. \hfill$\Box$

\bigskip

{\bf Remarks.} Our result on cycles (Theorem \ref{0}) is a generalization of the result of DeBiasio and Krueger (Theorem \ref{DKK}, bipartite version) to off-diagonal cases. It would be nice to obtain  a generalization of the result of Balogh, Kostochka, Lavrov, and Liu (Theorem \ref{BK}) to off-diagonal cases. Theorem \ref{0} gives an asymptotic result by establishing the exact result for connected matchings, it would be nice to get an exact result for cycles themselves.

%\nocite{*} %显示所有文献

\end{document}